\documentclass[a4paper]{amsart}
\usepackage{amsthm}
\usepackage[foot]{amsaddr}

\usepackage{tikz-cd}
\usepackage{stackrel,amssymb,amsmath}
\usepackage{graphicx}
\usepackage{stmaryrd}
\usepackage{enumitem}

\usetikzlibrary{calc}
\tikzset{curve/.style={settings={#1},to path={(\tikztostart)
    .. controls ($(\tikztostart)!\pv{pos}!(\tikztotarget)!\pv{height}!270:(\tikztotarget)$)
    and ($(\tikztostart)!1-\pv{pos}!(\tikztotarget)!\pv{height}!270:(\tikztotarget)$)
    .. (\tikztotarget)\tikztonodes}},
    settings/.code={\tikzset{quiver/.cd,#1}
        \def\pv##1{\pgfkeysvalueof{/tikz/quiver/##1}}},
    quiver/.cd,pos/.initial=0.35,height/.initial=0}

\usepackage[colorlinks=true]{hyperref}
\hypersetup{allcolors=[rgb]{0.1,0.1,0.4}}

\usepackage[square]{natbib}
\newcommand{\citey}{\citeyearpar}
\renewcommand{\cite}{\citep}

\title[A 2-categorical analysis of context comprehension]{A 2-categorical analysis\\of context comprehension}

\author[Coraglia]{Greta Coraglia}
\address{Università degli Studi di Milano, Via Festa del Perdono 7 - 20122 Milano}
\email{greta.coraglia@unimi.it}

\author[Emmenegger]{Jacopo Emmenegger}
\address{Università degli Studi di Genova, Via Dodecaneso 35 - 16146 Genova}
\email{emmenegger@dima.unige.it}

\thanks{%
We thank the anonymous referee for a careful reading.\\
The first author's research has been partially funded by the Project PRIN2020 ``BRIO'' (2020SSKZ7R) and by by the Department of Philosophy ``Piero Martinetti'' of the University of Milan under the Project ``Departments of Excellence 2023-2027'', both awarded by the Ministry of University and Research (MUR). The second author's research has been partially funded by the project ``PNRR - Young Researchers'' (SOE\_0000071) of the Italian Ministry of University and Research (MUR).
}

\keywords{dependent type theory, category with families, comprehension category, structure-semantics adjunction, 2-category theory, coalgebra}

\theoremstyle{definition}
\newtheorem{definition}{Definition}[section]

\theoremstyle{theorem}
\newtheorem{theorem}[definition]{Theorem}
\newtheorem{proposition}[definition]{Proposition}
\newtheorem{corollary}[definition]{Corollary}
\newtheorem{lemma}[definition]{Lemma}

\theoremstyle{remark}

\newtheorem{remark}[definition]{Remark}
\newtheorem{example}[definition]{Example}

\newcommand{\gencwf}[1]{generalised categor#1 with families}
\newcommand{\gcwf}{gcwf}
\newcommand{\gcwfs}{gcwfs}
\newcommand{\wccmd}[1][]{w-comonad#1}
\newcommand{\wccomonad}[1][]{weakening and contraction comonad#1}

\newcommand{\ie}{\textit{i.e.}\ }
\newcommand{\eg}{\textit{e.g.}\ }
\newcommand{\define}[1]{\emph{#1}}
\newcommand{\catof}[1]{\mathbf{#1}}
\newcommand{\ctg}[1]{\mathcal{#1}}
\newcommand{\car}{\rightarrowtriangle}
\newcommand{\cod}{\mathsf{cod}}
\newcommand{\dom}{\mathsf{dom}}
\newcommand{\id}{\mathrm{id}}
\newcommand{\Id}{\mathrm{Id}}
\newcommand{\idd}{\boldsymbol{id}}

\newcommand{\ctgid}[1]{\Id_{#1}}
\newcommand{\natiso}{\stackrel{\sim}{\Rightarrow}}
\newcommand{\due}{^{\mathsf{2}}}

\newcommand{\coal}{\mathrm{CoAlg}}
\newcommand{\emfrg}[1]{\mathrm{U}_{#1}}
\newcommand{\emrgt}[1]{\mathrm{R}_{#1}}
\newcommand{\emcmp}[1]{\mathrm{K}_{#1}}
\newcommand{\eminvcmp}[1]{\mathrm{J}_{#1}}

\newcommand{\opp}{^{\mathsf{op}}}
\newcommand{\cop}{^{\mathsf{co}}}

\newcommand{\Ty}{\;\texttt{Type}}
\newcommand{\Ctx}{\;\texttt{ctx}}
\newcommand{\ctxcat}{\mathsf{Ctx}}

\newcommand{\Cat}{\catof{Cat}}
\newcommand{\fibCat}{\catof{Fib}}

\newcommand{\jdtt}[1]{\mathbb{#1}}
\newcommand{\duu}{\dot{\mathcal{U}}}
\newcommand{\uu}{\mathcal{U}}
\newcommand{\du}{\dot{u}}

\newcommand{\ee}{\mathcal{E}}

\newcommand{\tand}{\quad\text{and}\quad}

\newcommand{\bb}{\ctg{B}}

\newcommand{\adj}[1]{\rotatebox{#1}{$\dashv$}}

\newcommand{\pseudosub}[1]{#1_{\mathbf{ps}}}
\newcommand{\strictsub}[1]{#1_{\mathbf{str}}}
\newcommand{\looseArg}{^{\cong}}

\newcommand{\cmdCatArg}{Cmd}
\newcommand{\cmdCat}{\catof{\cmdCatArg}}
\newcommand{\cmdpsCat}{\catof{\pseudosub{\cmdCatArg}}}
\newcommand{\cmdstrCat}{\catof{\strictsub{\cmdCatArg}}}
\newcommand{\adjLCatArg}{LAdj}
\newcommand{\adjLCat}{\catof{\adjLCatArg}}
\newcommand{\adjpsLCat}{\pseudosub{\adjLCat}}
\newcommand{\adjstrLCat}{\strictsub{\adjLCat}}
\newcommand{\adjLlCat}{\catof{\adjLCatArg}\looseArg}
\newcommand{\adjpsLlCat}{\pseudosub{\adjLlCat}}
\newcommand{\adjstrLlCat}{\strictsub{\adjLlCat}}
\newcommand{\adjRCat}{\catof{RAdj}}

\newcommand{\jdttCatArg}{GCwF}
\newcommand{\jdttCat}{\catof{\jdttCatArg}}
\newcommand{\jdttpsCat}{\catof{\pseudosub{\jdttCatArg\!}}}
\newcommand{\jdttstrCat}{\catof{\strictsub{\jdttCatArg\!}}}
\newcommand{\jdttlCat}{\catof{\jdttCatArg\looseArg}}
\newcommand{\jdttlpsCat}{\catof{\pseudosub{\jdttCatArg}\looseArg}}

\newcommand{\dscjdttCat}{\catof{Dsc\jdttCatArg}}
\newcommand{\dscjdttpsCat}{\catof{\pseudosub{Dsc\jdttCatArg}}}
\newcommand{\dscjdttstrCat}{\catof{\strictsub{Dsc\jdttCatArg}}}

\newcommand{\cwfCatArg}{CwF}
\newcommand{\cwfCat}{\catof{\cwfCatArg}}
\newcommand{\cwfpsCat}{\catof{\pseudosub{\cwfCatArg}}}
\newcommand{\cwfstrCat}{\catof{\strictsub{CwF}}}

\newcommand{\wccmdCatArg}{WCmd}
\newcommand{\wccmdCat}{\catof{\wccmdCatArg}}
\newcommand{\wccmdpsCat}{\catof{\pseudosub{\wccmdCatArg}}}
\newcommand{\wccmdstrCat}{\catof{\strictsub{\wccmdCatArg}}}

\newcommand{\compcatCatArg}{CompCat}
\newcommand{\compcatCat}{\catof{\compcatCatArg}}
\newcommand{\compcatpsCat}{\catof{\pseudosub{\compcatCatArg}}}
\newcommand{\compcatstrCat}{\catof{\strictsub{\compcatCatArg}}}
\newcommand{\dsccompcatCat}{\catof{Dsc\compcatCatArg}}
\newcommand{\dsccompcatpsCat}{\catof{\pseudosub{Dsc\compcatCatArg}}}

\newcommand{\splcompcatCat}{\catof{Spl\compcatCatArg}}
\newcommand{\splcompcatpsCat}{\catof{\pseudosub{Spl\compcatCatArg}}}

\newcommand{\fllcompcatpsCat}{\catof{\pseudosub{Full\compcatCatArg}}}

\newcommand{\fllsplcompcatpsCat}{\catof{\pseudosub{FullSpl\compcatCatArg}}}

\newcommand{\twofctr}[1]{\mathrm{#1}}

\newcommand{\cmdtoadj}{\twofctr{EM}}
\newcommand{\cmdtoadjloo}{\twofctr{EM^{\looseArg}}}
\newcommand{\adjtocmd}{\twofctr{C}}
\newcommand{\adjlootocmd}{\twofctr{C^{\looseArg}}}
\newcommand{\jdtttowccmd}{\twofctr{\hat{C}}}
\newcommand{\jdttpstowccmd}{\twofctr{\hat{C}{}^{\looseArg}}}
\newcommand{\wccmdtojdtt}{\twofctr{\hat{EM}}}
\newcommand{\wccmdtojdttps}{\twofctr{\hat{EM}{}^{\looseArg}}}
\newcommand{\wccmdtocmpct}{\twofctr{Y}}
\newcommand{\cmpcttowccmd}{\twofctr{X}}
\newcommand{\comptogcfw}{\twofctr{F}}
\newcommand{\gcwftocomp}{\twofctr{G}}

\newcommand{\unitadjcmd}{\boldsymbol{\eta}}
\newcommand{\unitadjloocmd}{\boldsymbol{\hat{\eta}}}

\begin{document}
\maketitle
\begin{abstract}
We consider the equivalence between the two main categorical models for the type-theoretical operation of context comprehension, namely P.~Dybjer's categories with families and B.~Jacobs' comprehension categories, and generalise it to the non-discrete case.
The classical equivalence can be summarised in the slogan: ``terms as sections''.
By recognising ``terms as coalgebras'', we show how to use the structure-semantics adjunction to prove that a 2-category of comprehension categories is biequivalent to a 2-category of (non-discrete) categories with families.
The biequivalence restricts to the classical one proved by Hofmann in the discrete case.
It also provides a framework where to compare different morphisms of these structures that have appeared in the literature, varying on the degree of preservation of the relevant structure.
We consider in particular morphisms defined by Claraimbault--Dybjer, Jacobs, Larrea, and Uemura.
\end{abstract}

\section{Introduction}\label{intro}
The problem of modelling the structural rules of type dependency using categories has motivated the study of several structures,
varying in generality, occurrence in nature, and adherence to the syntax of dependent type theory.
One aspect, that involving free variables and substitution,
is neatly dealt with using (possibly refinements of) Grothendieck fibrations.
The other main aspect of type dependency is the possibility of making assumptions
as encoded in the two rules below
\[
\frac{\Gamma\vdash A\Ty}{\vdash \Gamma.A\Ctx}
\hspace{4em}
\frac{\Gamma\vdash A\Ty}{\Gamma.A \vdash \texttt{v}_A : A}
\]
where the first one (\emph{context extension}) extends the context $\Gamma$ with the type $A$,
and the second one (\emph{assumption}) provides a ``generic term'' of $A$ in context $\Gamma.A$.
In the first order setting,
they allow us to add assumptions to a context,
and to prove what has been assumed, respectively.

The present paper provides a purely 2-categorical comparison of the two main categorical accounts of these two rules:
Jacobs' comprehension categories~\cite{jacobs1999categorical}
and Dybjer's categories with families~\cite{dybjer_inttt}.
They differ in that the former gives prominence to context extension,
and the latter to assumption.
For the comparison, a taxonomy of morphisms of both structures is proposed,
from lax versions to strict ones,
and a general biequivalence between 2-categories of lax morphisms is proved. This then specialises to (possibly stricter) equivalences between subcategories.

The taxonomy that we propose is based on the one,
well established,
for morphisms between comonads and between adjunctions~\cite{KellyStreetReview2cats}.
The fact that comprehension categories can be formulated as a pair of a fibration and a (suitable) comonad has been known since the early days.
In fact, Jacobs introduces these \emph{weakening and contraction comonads} first\footnote{Actually, they are introduced first in \cite{jacobs1999categorical}, but in the earlier \cite{comprehensioncats} they do not appear, in fact.} and uses them to justify comprehension categories~\cite[Definition~9.3.1, Theorem~9.3.4]{jacobs1999categorical}
(whose definition appears in a theorem).
We call them \emph{\wccmd[s]} for short.
On the other hand,
the formulation of categories with families as a pair of discrete fibrations over the same base connected by a (suitable) adjunction is also known,
but its formulation took some time and the observations (and proofs) of, among others, Fiore~\citey{fiore_att}, Awodey~\citey{awodey_2018}, and Uemura~\citey[Section~3]{Uemura2023}.
In order to have a uniform comparison with comprehension categories,
we drop the assumption of discreteness on the two fibrations
and call the resulting structure a \emph{\gencwf{y}},
see Definition~\ref{jdtt}, which has previously appeared in \cite{cjd} under a different name.

\begin{figure}[bh]
\caption{%
The underlying diagrams in $\Cat$ of, from left to right,
a comprehension category,
a \wccmd,
and a \gencwf{y}.
}
\[\begin{tikzcd}
\uu	\ar[dr,"u"'] \ar[rr,"{\chi}"]
&&	\ctg{B}^2	\ar[dl,"\mathrm{cod}"]
\\
&	\ctg{B}	&
\end{tikzcd}
\hspace{4em}
\begin{tikzcd}
\uu \ar[d,"u"] \ar[loop left]{l}{K}
\\
\ctg{B}
\end{tikzcd}
\hspace{4em}
\begin{tikzcd}
	{\duu} && {\uu} \\
	& {\ctg{B}}
	\arrow["{\dot{u}}"', from=1-1, to=2-2]
	\arrow["u", from=1-3, to=2-2]
	\arrow[""{name=0, anchor=center, inner sep=0}, "\Sigma"', curve={height=10pt}, from=1-1, to=1-3]
	\arrow[""{name=1, anchor=center, inner sep=0}, "\Delta"', curve={height=10pt}, from=1-3, to=1-1]
	\arrow["\dashv"{anchor=center, rotate=90}, draw=none, from=0, to=1]
\end{tikzcd}\]
\end{figure}
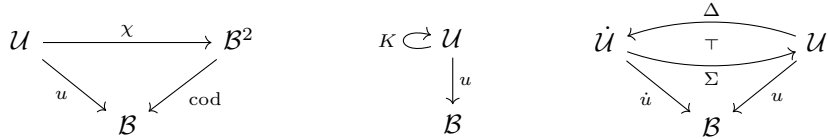

The correspondence between categories with families and comprehension categories is well-understood at the level of the objects.
Indeed, categories with families are in bijection with Cartmell's categories with attributes~\cite{CARTMELL1986209,moggi_1991},
which can be identified, via the Grothendieck construction,
with comprehension categories with discrete fibration.
The original proof is due to Hofmann~\citey[Section~3]{hofmann_1997},
and can be easily extended to an equivalence between categories of strict morphisms.
This is made explicit in~\cite{BlancoJ},
which provides a comprehensive investigation of the relations among several categories of structures for type dependency.
The morphisms considered there are, however, only \emph{strict} morphisms:
they preserve comprehensions on the nose.
If we wish to compare categories with families with structures not arising from syntax, strict morphisms are no longer useful.
As a case in point, consider Claraimbault and Dybjer's biequivalence between extensional type theories and locally cartesian closed categories~\cite{ClairambaultDybjer}.
Extensional type theories are there presented by certain categories with families with additional structure.
The morphisms between them needed to make the biequivalence work,
called \emph{pseudo cwf-morphisms},
are not the strict morphisms of cwfs defined by Dybjer~\citey{dybjer_inttt} and considered by Blanco~\citey{BlancoJ}.
In fact,
these pseudo cwf-morphisms are not morphisms of discrete fibrations,
and do not strictly preserve generic terms.
They are, however, morphisms of (certain) \gencwf{ies}.

Categories with families are in bijection with discrete comprehension categories because,
for every object $A$ of $\uu$,
the objects of $\duu$ mapped to $A$ (the terms) are in bijection with sections of the display map $\chi A$.
But sections can be describes as coalgebras,
and these sections are the coalgebras of the \wccmd{} $K$ induced by $\chi$.
This simple observation suggests that
the classical correspondence between categories with families and comprehension categories
could be phrased within the framework of the correspondence between adjunctions and comonads.
The internal structure-semantics adjunction~\cite{street72formnd}
can be used to show that comonads are 2-reflective in a suitable 2-category of adjunctions,
where the 1-cells are pairs of functors commuting with the left adjoints.
Of course, this reflection is in general far from being an equivalence.
Nevertheless, we show that it lifts to a 2-reflection between \gencwf{ies} and \wccmd[s]
which becomes a biequivalence
if one takes as morphisms of \gencwf{ies}
functors that commute with left adjoints up to a natural vertical iso.
We call these \emph{loose} morphisms.
In type theoretic terms, this means preserving typing only up to iso.
The discrete case is recovered thanks to the fact that
vertical isos in discrete fibrations are identities.

Section~\ref{sec:cmd-adj}
reviews the taxonomy of morphisms of (co)monads and adjunctions
and the details of the 2-reflection between them~\eqref{thm:cmd-adj},
and extends it to the case of loose morphisms of adjunctions in~\eqref{thm:cmd-adjloo}.
Section~\ref{sec:defs}
defines the 2-categories of interest:
in order, those of comprehension categories in~\eqref{cmpct-2cats},
of \wccmd[s] in~\eqref{2catofwccmd},
and of \gencwf{ies} in~\eqref{jdtt-2cat}.
Section~\ref{sec:bieqv}
proves the biequivalence of comprehension categories and \gencwf{ies}
by establishing first a biequivalence between comprehension categories and \wccmd[s] in~\eqref{eqv-wccmd-cmpct}, 
and then a biequivalence between the latter and \gencwf{ies} in~\eqref{eqv-wccmd-jdtt}.
We conclude in~\eqref{sec:disc-full-compcats} considering the discrete case,
and the case of full comprehension categories.

\section{A biadjunction between comonads and {adjunctions}}
\label{sec:cmd-adj}

Every adjunction $L \dashv R$ determines a comonad on the composite $LR$ (and a monad on $RL$),
as it was observed in~\cite{huber61}.
Conversely, every comonad determines an adjunction
via the Eilenberg--Moore construction of the category of coalgebras~\cite{em65}.
In fact it determines two adjunctions---%
the second one being given by the Kleisli construction~\cite{kleisli65} of the category of free algebras,
but we shall only be interested in the former.
As it turns out, the Eilenberg--Moore construction provides a fully faithful embedding of comonads into adjunctions,
with a reflector given by the comonad induced by an adjunction.
This can be seen restricting (and dualising) the classical structure-semantics adjunction~\cite{Dubuc1970}.
In this section we shall recall some details of this construction,
which we need to show that it extends to a 2-reflection,
and that it further extends to the case of \emph{loose} morphisms of adjunctions.

\subsection{Morphisms of adjunctions and of comonads}

The 2-category of comonads can be defined as a suitable dual of a 2-category of formal monads.
We refer to the original source~\cite{street72formnd} for what we need of the theory of formal monads in a 2-category.

Given a 2-category $\mathbf{C}$,
we write $\mathbf{C}^{\mathrm{op}}$ for the 2-category with the 1-cells reversed,
and $\mathbf{C}^{\mathrm{co}}$ for the 2-category with the 2-cells reversed.

\begin{definition}\label{def:cmdCat}
The 2-category $\cmdCat$ is defined as $\catof{Mnd}(\Cat^{\mathrm{co}})^{\mathrm{co}}$,
where $\catof{Mnd}(\mathbf{X})$ denotes the 2-category of formal monads in a 2-category $\mathbf{X}$.
The definition unfolds as follows.

A \define{0-cell} is a pair of a category $\mathcal{C}$ and a comonad $(K,\epsilon,\nu)$ on $\mathcal{C}$.

A \define{1-cell} from $(\mathcal{C},K,\epsilon,\nu)$ to
$(\mathcal{C}',K',\epsilon',\nu')$ is a (lax) morphism of comonads,
that is, a pair $(H,\theta)$ of a functor
$H \colon \mathcal{C} \to \mathcal{C}'$
and a natural transformation $\theta \colon HK \Rightarrow K'H$
such that the diagrams below commute.
\[\begin{tikzcd}[row sep=2em,column sep=2em]
HK	\ar[dr,Rightarrow,"H\epsilon"{swap}] \ar[rr,Rightarrow,"\theta"]
&&	K'H	\ar[dl,Rightarrow,"\epsilon' H"]
\\
&	H	&
\end{tikzcd}
\hspace{3em}
\begin{tikzcd}[row sep=2em,column sep=2em]
HK	\ar[d,Rightarrow,"H\nu"{swap}] \ar[rr,Rightarrow,"\theta"]
&&	K'H	\ar[d,Rightarrow,"\nu'H"]
\\
HK^2	\ar[r,Rightarrow,"\theta K"{swap}]
&	K'HK	\ar[r,Rightarrow,"K'\theta"{swap}]
&	K'^2H
\end{tikzcd}\]

The \define{composite} of two (composable) morphisms of comonads $(H_1,\theta_1)$ and $(H_2,\theta_2)$
is $(H_2H_1,(\theta_2H_1)(H_2\theta_1))$.

A \define{2-cell} from $(H_1,\theta_1)$ to $(H_2,\theta_2)$ is a natural transformation $\phi \colon H_1 \Rightarrow H_2$ such that
$(K'\phi)\theta_1 = \theta_2(\phi K)$.

A morphism of comonads $(H,\theta)$ is a \emph{pseudo} (respectively, \emph{strict}) morphism
if $\theta$ is invertible (respectively, the identity).
The identity morphism is strict,
and it is clear that pseudo and strict morphisms are closed under composition.
We write $\cmdpsCat$ and $\cmdstrCat$ for the 2-full 2-subcategories of $\cmdCat$ with the same 0-cells, and only those 1-cells $(H,\theta)$
which are pseudo (respectively, strict) morphisms of comonads.
\end{definition}

\begin{remark}\label{lifttocoal}
The right-hand diagram in the definition of lax morphism of comonads~\eqref{def:cmdCat}
can be read as saying that,
given a lax morphism of comonads
${(H,\theta) \colon}$ ${(K,\epsilon,\nu) \to (K',\epsilon',\nu')}$,
each component $\theta_E$ is a morphism of coalgebras
from $(HKE,$ $\theta_{KE}\circ H\nu_E)$ to $(K'HE,\nu'_{HE})$.
This means that $\theta$ lifts to $\hat{\theta}$ below,
in the sense that $\emfrg{K'}\hat{\theta} = \theta$.
\[\begin{tikzcd}[column sep=4em,row sep=3em]
\mathcal{C}	\ar[d,"{\emrgt{K}}"',"\ "{pos=.5,name=dom}] \ar[r,"H"]
&	\mathcal{C}'	\ar[d,"{\emrgt{K'}}","\ "'{pos=.5,name=cod}]
\\
\coal(K)	\ar[r,"{\coal(H,\theta)}"']
&	\coal(K')	\ar[from=dom,to=cod,Rightarrow,"\hat{\theta}"]
\end{tikzcd}\]
\end{remark}

Several kinds of morphisms between adjunctions can be considered.
The list below is compiled from the squares of one of the two double categories of adjunctions defined in~\cite[pg.~86]{KellyStreetReview2cats}.
In particular, all these morphisms have unital and associative compositions.
The double category defined by Kelly and Street consists of:
objects are categories,
vertical morphisms given by adjunctions, directed according to the left adjoint;
horizontal morphisms given by functors;
squares given by natural transformations filling the square involving left adjoints,
as in the left-hand square below.
\begin{equation}\label{double-adj}
\begin{tikzcd}[column sep=4em,row sep=2.5em]
\ctg{D}	\ar[d,"L"',"\ "{pos=.5,name=cd}] \ar[r,"G"]
&	\ctg{D}'	\ar[d,"L'","\ "'{pos=.5,name=dm}]
\\
\ctg{C}	\ar[r,"F"']
&	\ctg{C}'
\ar[from=dm,to=cd,Rightarrow,"\zeta"']
\end{tikzcd}
\hspace{5em}
\begin{tikzcd}[column sep=4em,row sep=2.5em]
\ctg{D}	\ar[r,"G"]	&	\ctg{D}'
\\
\ctg{C}	\ar[u,"R","\ "'{pos=.5,name=dm}] \ar[r,"F"']
&	\ctg{C}'	\ar[u,"R'"',"\ "{pos=.5,name=cd}]
\ar[from=dm,to=cd,Rightarrow,"\xi"']
\end{tikzcd}
\end{equation}
The other one is defined similarly, but using right adjoints instead
as in the right-hand square above.
These two double categories are isomorphic~\cite[Proposition~2.2]{KellyStreetReview2cats}.
The isomorphism is the identity on everything but 2-cells,
and maps 2-cells $\zeta \colon L'G \Rightarrow FL$ and $\xi \colon GR \Rightarrow R'F$
to their \define{mates}
$\zeta^\# := (R'F\epsilon)(R'\zeta R)(\eta'GR)$
and $\xi^\# := (R'F\epsilon)(R'\zeta R)(\eta'GR)$,
as shown below in~\eqref{def:mate}.
\begin{equation}\label{def:mate}
\begin{tikzcd}[column sep=6ex, row sep=5ex]
\ctg{C}	\ar[dr,bend right=2em,"\Id"',"\ "{pos=.4,name=epcd}] \ar[r,"R"]
&	\ctg{D}		\ar[d,"L"{description},"\ "'{pos=.55,name=epdm},"\ "{pos=.55,name=zcd}] \ar[r,"G"]
&	\ctg{D}'	\ar[d,"L'"{description},"\ "{pos=.55,name=etcd},"\ "'{pos=.55,name=zdm}]
	\ar[dr,bend left=2em,"\Id",""'{pos=.65,name=etdm}]
&\\
&	\ctg{C}		\ar[r,"F"']
&	\ctg{C}'	\ar[r,"R'"']
&	\ctg{D}'
\ar[from=zdm,to=zcd,Rightarrow,"\zeta"']
\ar[from=etdm,to=etcd,Rightarrow,"\eta'"']
\ar[from=epdm,to=epcd,Rightarrow,"\epsilon"']
\end{tikzcd}
\hspace{8ex}
\begin{tikzcd}[column sep=6ex, row sep=5ex]
&	\ctg{D}	\ar[r,"G"]
&	\ctg{D}' \ar[r,"L'"]
&	\ctg{C}'
\\
\ctg{D}	\ar[r,"L"'] \ar[ur,bend left=2em,"\Id","\ "'{pos=.43,name=etdm}]
&	\ctg{C}	\ar[u,"R"'{description},"\ "'{pos=.55,name=dm},"\ "{pos=.55,name=etcd}] \ar[r,"F"']
&	\ctg{C}'	\ar[u,"R'"'{description},"\ "'{pos=.55,name=epdm},"\ "{pos=.55,name=cd}]
\ar[ur,bend right=2em,"\Id"',"\ "{pos=.68,name=epcd}]
\ar[from=dm,to=cd,Rightarrow,"\xi"']
\ar[from=etdm,to=etcd,Rightarrow,"\eta"']
\ar[from=epdm,to=epcd,Rightarrow,"\epsilon'"']
\end{tikzcd}
\end{equation}

As we are not interested in composing adjunctions,
we take adjunctions as objects and consider the squares,
\ie the triples consisting of two functors and the natural transformation, as morphisms.
Moreover, we shall only be interested in those squares whose transformation is invertible.

\begin{definition}\label{def:adjmorph}
Let $(L,R,\eta,\epsilon)$ and $(L',R',\eta',\epsilon')$ be adjunctions,
where $L \colon \mathcal{D} \to \mathcal{C}$ and $L' \colon \mathcal{D}' \to \mathcal{C}'$.

A \define{left loose morphism of adjunctions}
from $(L,R,\eta,\epsilon)$ to $(L',R',\eta',\epsilon')$ is a triple $(F,G,\zeta)$
where $F \colon \mathcal{C} \to \mathcal{C}'$
and $G \colon \mathcal{D} \to \mathcal{D}'$ are functors,
and $\zeta \colon L'G \natiso FL$ is a natural iso.

A \define{right loose morphism of adjunctions}
is defined dually as a triple $(F,G,\xi)$
where $F \colon \mathcal{C} \to \mathcal{C}'$
and $G \colon \mathcal{D} \to \mathcal{D}'$ are functors,
and $\xi \colon GR \natiso R'F$ is a natural iso.

A \define{left morphism of adjunctions} is a left loose morphism of adjunctions with $\zeta = \id$.
In particular, in this case $L'G=FL$.
Similarly, a \define{right morphism of adjunctions} is a right loose morphism of adjunctions with $\xi = \id$.

A left loose morphism of adjunctions $(F,G,\zeta)$
is a \define{pseudo left loose morphism} if the mate $\zeta^\#$ is invertible.
It is a \define{strict left loose morphism} if the mate $\zeta^\#$ is the identity.
In particular, in this case $GR = R'F$.
\end{definition}

\begin{remark}\label{form-adj}
It follows from~\cite[Proposition~12]{melliesrolland2020}
that a left loose morphism of adjunctions $(F,G,\zeta)\colon (L,R) \to (L',R')$
gives rise to a formal adjunction in the 2-category $\catof{Fun_{ps}}$,
which consists of functors, squares commuting up to a natural iso, and pairs of compatible natural transformations.
The formal adjoints are the two 1-cells
$(L,L',\zeta^{-1})\colon(\ctg{D},\ctg{D}',G)\to(\ctg{C},\ctg{C},F)$
and $(R,R',\zeta^{\#})\colon(\ctg{C},\ctg{C},F)\to(\ctg{D},\ctg{D}',G)$.
However, note that our objects are adjunctions whereas their objects are functors:
this setting seems somewhat orthogonal to ours and not easily comparable.
\end{remark}

The composite of two (composable) left loose morphisms of adjunctions
$(F_1,G_1,\zeta_1)$ and $(F_2,G_2,\zeta_2)$
is $(F_2F_1,G_2G_1,(F_2\zeta_1)(\zeta_2G_1))$.
It follows from~(\ref{rem:adjmorph}.\ref{rem:adjmorph:matecmp}) below
that pseudo and strict left loose morphisms are closed under composition.
The same is true for right loose morphisms.

\begin{remark}\label{rem:adjmorph}
Using naturality of the arrows involved and the triangular identities,
it is straightforward to verify the following facts
about mates of left loose morphisms.
\begin{enumerate}
\item\label{rem:adjmorph:matecmp}
Let $(F_1,G_1,\zeta_1)$ and $(F_2,G_2,\zeta_2)$
be two composable left loose morphisms of adjunctions.
Then
\[
(F_2\zeta_1 \circ \zeta_2 G_1)^\# = \zeta_2^\#F_1 \circ G_2\zeta_1^\#.
\]
\item\label{rem:adjmorph:onecell}
Let $(F,G,\zeta)$ be a left loose morphism of adjunctions.
Then the two squares below commute.
\[\begin{tikzcd}[row sep=2em,column sep=2em]
G	\ar[d,Rightarrow,"G\eta"{swap}] \ar[r,Rightarrow,"\eta'G"]
&	R'L'G	\ar[d,Rightarrow,"R'\zeta"]
\\
GRL	\ar[r,Rightarrow,"\zeta^\#L"{swap}]	&	R'FL
\end{tikzcd}
\hspace{5em}
\begin{tikzcd}[row sep=2em,column sep=2em]
L'GR	\ar[d,Rightarrow,"L'\zeta^\#"{swap}] \ar[r,Rightarrow,"\zeta R"]
&	FLR	\ar[d,Rightarrow,"F\epsilon"]
\\
L'R'F	\ar[r,Rightarrow,"\epsilon'F"{swap}]	&	F
\end{tikzcd}\]
\item\label{rem:adjmorph:twocell}
Consider two left loose morphisms of adjunctions
$(F_1,G_1,\zeta_1)$ to $(F_2,G_2,\zeta_2)$
and a pair $(\phi,\psi)$ of natural transformations
$\phi \colon F_1 \Rightarrow F_2$ and $\psi \colon G_1 \Rightarrow G_2$.
Then the left-hand square below commutes if and only if the right-hand one does.
\[\begin{tikzcd}[row sep=2em,column sep=2em]
L'G_1	\ar[d,Rightarrow,"L'\psi"'] \ar[r,Rightarrow,"{\zeta_1}"]
&	F_1L	\ar[d,Rightarrow,"\phi L"]
\\
L'G_2	\ar[r,Rightarrow,"\zeta_2"']
&	F_2L
\end{tikzcd}
\hspace{5em}
\begin{tikzcd}[row sep=2em,column sep=2em]
G_1R	\ar[d,Rightarrow,"\psi R"'] \ar[r,Rightarrow,"{\zeta_1^\#}"]
&	R'F_1	\ar[d,Rightarrow,"R'\phi"]
\\
G_2R	\ar[r,Rightarrow,"{\zeta_2^\#}"']
&	R'F_2
\end{tikzcd}\]
\end{enumerate}
\end{remark}

\begin{definition}\label{def:adjCat}
The 2-category $\adjLlCat$ is defined as follows.

A \define{0-cell} is an adjunction.

A \define{1-cell} is a left loose morphism of adjunctions.

A \define{2-cell} from $(F_1,G_1,\zeta_1)$ to $(F_2,G_2,\zeta_2)$
is a pair $(\phi,\psi)$ of natural transformations
$\phi \colon F_1 \Rightarrow F_2$ and $\psi \colon G_1 \Rightarrow G_2$
such that the left-hand square above in (\ref{rem:adjmorph}.\ref{rem:adjmorph:twocell}) commutes.
Pasting squares, we see that the 2-cells are closed under component-wise composition.

The 2-category $\adjLCat$ is the 2-full sub-2-category of $\adjLlCat$ on the 1-cells which are left morphisms of adjunctions.
Here the 2-cells are pairs $(\phi,\psi)$ such that $L'\psi=\phi L$.

We write $\adjpsLlCat$, $\adjpsLCat$, $\adjstrLlCat$, and $\adjstrLCat$
for the 2-full sub-2-categories on pseudo and strict left (loose) morphisms,
respectively.
In the last two cases the 2-cells are pairs $(\phi,\psi)$ such that $R'\phi=\psi R$.

The 2-categories $\adjRCat^{\cong}$ and $\adjRCat$ are defined similarly to $\adjLlCat$ and $\adjLCat$, respectively,
but using right (loose) morphisms instead of left ones.
\end{definition}

In other words, $\adjRCat^{\cong}$ and $\adjLlCat$ are the categories of vertical arrows and pseudo squares
of the two double categories of adjunctions defined in~\cite[pg.~86]{KellyStreetReview2cats},
together with the 2-cells defined above.
Actually, Kelly and Street work in a (suitable) 2-category $\mathbf{X}$,
and define adjunctions and morphisms of them internally to $\mathbf{X}$.
From this more general perspective, it is possible to observe that
$\adjRCat(\mathbf{X}) = \adjLCat(\mathbf{X}\cop)\cop$ and, dually,
$\adjLCat(\mathbf{X}) = \adjRCat(\mathbf{X}\cop)\cop$.
The same holds of course for $\adjLlCat(\mathbf{X})$ and $\adjRCat^{\cong}(\mathbf{X})$
(and for the 2-categories whose morphisms are the squares of Kelly and Street's double categories).

\subsection{The 2-reflection between comonads and adjunctions}
\label{sec:cmdadj-refl}

Let $\mathbf{X}$ be a 2-category that admits the construction of algebras.
The 2-category of formal monads on some object $X$ in $\mathbf{X}$
is 2-reflective in a suitable full sub-2-category of $\mathbf{X}/X$:
this is the content of the (internal) structure-semantics adjunction~\cite[Theorem~6]{street72formnd},
see~\cite{Dubuc1970} for the enriched version.
The construction of algebras for a monad $t$ in $\mathbf{X}$ provides a ``forgetful'' 1-cell $X^t \to X$,
where $X^t$ is the object of algebras:
this is the semantic functor.
The subcategory of $\mathbf{X}/X$ consists of those 1-cells $A \to X$, called \emph{tractable},
that induce a monad on $X$ (in a suitably universal way):
this is the structure functor.
The reflection is based on the observation that,
given a 1-cell $f \colon A \to X$ and a monad $t \colon X \to X$,
1-cells $g \colon A \to X^t$ over $X$
are in bijection with algebra structures on $f$,
that is, 2-cells $\psi \colon tf \Rightarrow f$
making the usual diagrams involving unit and multiplication commute
(this is obvious when $\mathbf{X}=\Cat$,
and follows from the universal property of $X^t$ in general).
But then the pair $(f,\psi)$ is precisely a (lax) morphism of monads from the identity monad on $A$ to $t$.
The definition of tractable functor ensures that
these are in bijection with (lax) morphisms of monads from the monad induced by $f$,
called \emph{codensity monad}, to $t$.\footnote{%
In Dubuc the adjunction involves the opposite of the category of monads over a fixed category $\ctg{B}$:
this is because his morphisms of monads are the \emph{oplax} ones, instead of lax ones.
Over $\ctg{B}$ it is enough to take the opposite since
$(\catof{Mnd}_{\ctg{B}}^{\mathrm{oplax}})\opp = \catof{Mnd}_{\ctg{B}}^{\mathrm{lax}}$.
However, this is no longer true if we do not work over a fixed base.
}
In practice, tractable functors can be defined via right Kan extensions~\cite{Dubuc1970},
or via cocartesian lifts~\cite{street72formnd}, but we do not need the precise definition.
For us, it is enough to observe that right adjoint functors are tractable:
in this case the codensity monad is the monad induced by the adjunction.
More precisely,
let us consider the 2-category $\catof{Radj}(\mathbf{X})$,
which is defined as $\adjRCat$ in~\eqref{def:adjCat} but internally to $\mathbf{X}$.
Its sub-2-category $\catof{Radj}(\mathbf{X})_X$ on the adjunction whose right adjoint has codomain $X$ (and both 1- and 2-cells are identities) embeds fully into the full sub-2-category of $\mathbf{X}/X$ on the tractable functors.
Since the semantics functor clearly lands in $\catof{Radj}(\mathbf{X})_X$,
the 2-reflection restricts between $\catof{Mnd}(\mathbf{X})_X$ and $\catof{Radj}(\mathbf{X})_X$.
Moreover, the family of 2-reflections extends to a 2-reflection between the global 2-categories of monads and adjunctions over the 2-category of arrows of $\mathbf{X}$, as shown in~\eqref{mndadj-refl} below.

\begin{equation}\label{mndadj-refl}
\begin{tikzcd}
\catof{Mnd}(\mathbf{X})	\ar[dr,bend right=1em,hook,"\mathrm{EM}"']
	\ar[rr,shift right=1.3ex,hook,"{\mathrm{EM}}"',""{name=right}]
&&	\catof{Radj}(\mathbf{X})	\ar[dl,bend left=1em,hook']
	\ar[ll,shift right=1.3ex,"{\mathrm{M}}"',""{name=left}]
\\
&	\mathbf{X^2}	&	\ar[from=right,to=left,phantom,"\adj{-90}"{description}]
\end{tikzcd}
\end{equation}

By taking $\mathbf{X} = \Cat^{\mathrm{co}}$ and  recalling that
$\cmdCat = \catof{Mnd}(\Cat\cop)\cop$ and $\adjLCat = \adjRCat(\Cat\cop)\cop$,
we see that the 2-category $\cmdCat$ is a 2-reflective sub-2-category of $\adjLCat$.
It is also straightforward to verify that the reflection restricts to the sub-2-categories of pseudo and strict morphisms.
We record this fact in the theorem below.

\begin{theorem}\label{thm:cmd-adj}
There is a 2-reflection
\[\begin{tikzcd}[column sep=6em]
\cmdCat	\ar[r,shift right=1ex,hook,"\cmdtoadj"{swap}] \ar[r,phantom,shift left=.3ex,"\adj{-90}"{description}]
&	\adjLCat	\ar[l,shift right=1.5ex,"\adjtocmd"{swap}]
\end{tikzcd}\]
such that the counit is the identity $\adjtocmd \circ \cmdtoadj = \Id_{\cmdCat}$.
In particular,
the right adjoint $\cmdtoadj$ is injective on objects and fully faithful.

The 2-reflection restricts between the wide 2-full sub-2-categories on the pseudo and strict morphisms.
\end{theorem}

In fact, the 2-reflection in Theorem~\ref{thm:cmd-adj}
can be extended to a bireflection involving the 2-category
whose 1-cells are left loose morphisms of adjunctions.
This is not hard to see,
but we find it convenient to first recall some details of the proof of~\eqref{thm:cmd-adj}.
These details will also be helpful in clarifying the proof of our main result in~\eqref{sec:bieqv}.

\subsubsection{The right adjoint $\cmdtoadj$}
\label{sec:cmd-adj:ra}

The 2-functor $\cmdtoadj$ maps a comonad $(K,\epsilon,\nu)$
to the Eilenberg--Moore adjunction of coalgebras~\cite[VI.3]{CWM}:
\[\begin{tikzcd}[column sep=5em]
\coal(K)	\ar[r,shift right=1ex,"{\emfrg{K}}"{swap}] \ar[r,phantom,shift left=.3ex,"\adj{90}"{description}]
&	\mathcal{C}	\ar[l,shift right=1.5ex,"{\emrgt{K}}"{swap}]
\end{tikzcd}\]
whose counit is $\epsilon \colon \emfrg{K} \emrgt{K} =K \Rightarrow \Id_{\mathcal{C}}$
and whose unit $\eta^K \colon \Id_{\coal(K)} \Rightarrow \emrgt{K}\emfrg{K}$
has component at a coalgebra $(A,a \colon A \to KA)$
the arrow $a$ itself seen as a morphism of coalgebras $(A,a) \to (KA,\nu_A)$.

A lax morphism of comonads $(H,\theta) \colon K \to K'$
induces a functor $\coal(H,\theta)$ from $\coal(K)$ to $\coal(K')$
which maps a $K$-coalgebra $(A,a)$ to the $K'$-coalgebra $(HA,\theta_A \circ Ha)$.
Clearly, $\emfrg{K'} \coal(H,\theta) = H \emfrg{K}$.
Therefore the pair $(H,\coal(H,\theta))$ is a left morphism of adjunctions,
which gives the action of the 2-functor $\cmdtoadj$ on 1-cells.

Finally, it is easy to see that every 2-cell $\phi$ in $\cmdCat$ lifts to a natural transformation
$\coal(\phi) \colon \coal(H_1,\theta_1) \Rightarrow \coal(H_2,\theta_2)$ whose component at $(A,a)$ is $\phi_A$ itself.
Therefore $(\phi, \coal(\phi))$ is a 2-cell in $\adjLCat$,
which gives the action of $\cmdtoadj$ on 2-cells.

It is straightforward to verify that the mate of
$\id \colon \emfrg{K'}\coal(H,\theta) = H\emfrg{K}$ is $\theta$ itself.
It follows that the functor $\cmdtoadj$ restricts to the sub-2-categories on pseudo and strict morphisms.

\subsubsection{The 2-reflector $\adjtocmd$}
\label{sec:cmd-adj:la}

The 2-functor $\adjtocmd$ maps an adjunction $(L,R,\eta,\epsilon)$ to the comonad $(LR,\epsilon,L\eta R)$.

A left morphism of adjunctions $(F,G)$ induces a lax morphism of comonads
$\adjtocmd(F,G) = (F,L'\id^\#)$,
as we will see in~\eqref{ext-loose}.
It is then clear that $\adjtocmd$ restricts to the sub-2-categories on pseudo and strict morphisms.

A 2-cell $(\phi,\psi)$ in $\adjLCat$ is simply mapped to $\phi$.
A proof that this gives a 2-cell in $\cmdCat$ is in~\eqref{ext-loose}.

\subsubsection{The counit}
\label{sec:cmd-adj:counit}

We have
\[
\adjtocmd \circ \cmdtoadj = \Id.
\]
On objects, this follows from
$\emfrg{K}\emrgt{K} = K$ and $\emfrg{K}\eta\emrgt{K} = \nu$.
To see that $\adjtocmd \circ \cmdtoadj(H,\theta) = (H,\theta)$
for a lax morphism of comonads $(H,\theta)$,
recall that $\eta'_{(A,a)} = a$ and use the two diagrams in~\eqref{def:cmdCat}
to show that $\emfrg{K'}\id^\#_E = K'H\epsilon_E \circ \emfrg{K'}(\theta_{KE}\circ H\nu_E)$ equals $\theta_E$.
Finally, both functors act identically on 2-cells.

\subsubsection{The unit}
\label{sec:cmd-adj:unit}

Every adjunction $(L,R,\eta,\epsilon)$ gives rise
to a canonical comparison functor $\emcmp{L,R}$ making the diagram below commute.
\[\begin{tikzcd}
\mathcal{D}	\ar[dr,"L"{swap}] \ar[rr,"{\emcmp{L,R}}"]
&&	\coal(LR)	\ar[dl,"\emfrg{LR}"]
\\
&	\mathcal{C}	&
\end{tikzcd}\]
Recall that $\emcmp{L,R}$ maps an object $A$ to the coalgebra
$L\eta_A \colon LA \to LRLA$.

The unit $\unitadjcmd$ of the 2-adjunction $\adjtocmd \dashv \cmdtoadj$ at $(L,R,\eta,\epsilon)$
is defined as the strict left morphism of adjunctions
$(\Id,\emcmp{L,R}) \colon (L,R,\eta,\epsilon) \to (\emfrg{LR},\emrgt{LR},\eta^{LR},\epsilon)$.
This family is natural in $(L,R,\eta,\epsilon)$ since,
for every left morphism  of adjunctions $(F,G)$,
\begin{align*}
L'\id^\#_{LA} \circ FL\eta_A
&= L'R'F\epsilon_{LA} \circ L'\eta'_{GRLA} \circ L'G\eta_A
\\&= L'R'F\epsilon_{LA} \circ L'R'L'G\eta_A \circ L'\eta'_{GA}
\\&= L'\eta'_{GA}
\end{align*}
and $FLf=L'Gf$ for $A$ and $f$ in $\mathcal{D}$, imply
$\coal(\adjtocmd(F,G,\zeta)) \circ \emcmp{L,R} = \emcmp{L',R'} \circ G$.
Note that this proof heavily relies on $L'G=FL$.

\subsubsection{The trianguar identities}
\label{sec:cmd-adj:ti}

The two equations below hold.
\begin{equation}\label{eq:cmd-adj:triangid}
\adjtocmd\unitadjcmd = \idd_{\adjtocmd}
\hspace{6em}
\unitadjcmd\cmdtoadj = \idd_{\cmdtoadj}
\end{equation}
The left-hand one does since
the mate $\id^\# \colon \emcmp{L,R}R \Rightarrow \emrgt{LR}$ of
$\id \colon \emfrg{LR}\emcmp{L,R} \Rightarrow L$
is itself an identity.
The right-hand one does since
$\emcmp{\emfrg{K},\emrgt{K}} = \Id_{\coal(K)}$

Now we turn to the case of left loose morphisms of adjunctions.

\subsection{The bireflection for loose morphisms}
\label{sec:cmdadj-refl-loose}

\begin{lemma}\label{lem:extadjps}
Let $(F,G,\zeta) \colon (L,R,\eta,\epsilon) \to (L',R',\eta',\epsilon')$
be a left loose morphism of adjunctions.
Then the following facts hold.
\begin{enumerate}
\item\label{lem:extadjps:onecell}
The two diagrams below commute.
\[\begin{tikzcd}[row sep=2em,column sep=2em]
FLR	\ar[dr,Rightarrow,"F\epsilon"{swap}]
	\ar[rr,Rightarrow,"{L'\zeta^\# \circ \zeta^{-1}R}"]
&&	L'R'F	\ar[dl,Rightarrow,"\epsilon'F"]
\\
&	F	&
\end{tikzcd}
\hspace{3ex}
\begin{tikzcd}[row sep=2ex,column sep=2ex]
FLR	\ar[d,Rightarrow,"FL\eta R"{swap}]
	\ar[rr,Rightarrow,"{L'\zeta^\# \circ \zeta^{-1}}R"]
&&	L'R'F	\ar[d,Rightarrow,"L'\eta'R'F"]
\\[1ex]
F(LR)^2	\ar[dr,Rightarrow,"{(L'\zeta^\# \circ \zeta^{-1}R)LR}"{swap}]
&&	(L'R')^2F
\\
&	L'R'FLR	\ar[ur,Rightarrow,"{L'R'(L'\zeta^\# \circ \zeta^{-1}R)}"{swap}]
&\end{tikzcd}\]
\item\label{lem:extadjps:unitpsnat}
\[\begin{tikzcd}[column sep=3em,row sep=6ex]
L'G	\ar[d,Rightarrow,"{L'\eta'G}"'] \ar[r,Rightarrow,"{\zeta}"]
&	FL	\ar[d,Rightarrow,"{L'\zeta^\#L \circ \zeta^{-1}RL \circ FL\eta}"{description}]
\\
L'R'L'G	\ar[r,Rightarrow,"{L'R'\zeta}"']
&	L'R'FL
\end{tikzcd}\]
\end{enumerate}
\end{lemma}

\begin{proof}
Using (\ref{rem:adjmorph}.\ref{rem:adjmorph:onecell})
and naturality of the arrows involved. 

\ref{lem:extadjps:onecell}.
\[
\epsilon'F \circ L'\zeta^\# \circ \zeta^{-1}R
= F\epsilon \circ \zeta R \circ \zeta^{-1}R
= F\epsilon
\]
\begin{align*}
L'R'(L'\zeta^\# \circ \zeta^{-1} R) \circ L'\zeta^\#LR \circ
& \zeta^{-1}(RL)R \circ (FL)\eta R
\\&=
L'R'(L'\zeta^\# \circ \zeta^{-1} R) \circ L'(\zeta^\#L)R \circ L'(G\eta)R \circ \zeta^{-1}R
\\&=
L'R'(L'\zeta^\# \circ \zeta^{-1} R) \circ L'R'\zeta R \circ L'\eta'GR \circ \zeta^{-1}R
\\&=
L'(R'L')\zeta^\# \circ L'\eta'(GR) \circ \zeta^{-1}R
\\&=
L'\eta'R'F \circ L'\zeta^\# \circ \zeta^{-1}R
\end{align*}

\ref{lem:extadjps:unitpsnat}.
$L'\zeta^\#L \circ \zeta^{-1}RL \circ FL\eta \circ \zeta
= L'\zeta^\#L \circ L'G\eta
= L'R'\zeta \circ L'\eta'G.$
\end{proof}

\begin{corollary}\label{ext-loose}
The 2-functor $\adjtocmd$ extends along $\adjLCat \hookrightarrow \adjLlCat$
to a 2-functor $\adjlootocmd \colon \adjLlCat \to \cmdCat$ by defining
\begin{equation}\label{ext-loose-fnct}
\adjlootocmd(F,G,\zeta) := (F,L'\zeta^\# \circ \zeta^{-1}R)
\end{equation}
on 1-cells $(F,G,\zeta)\colon(L,R,\eta,\epsilon) \to (L',R',\eta',\epsilon')$.

This functor restricts between the sub-2-categories on pseudo morphisms.
Note that $\adjlootocmd$ restricted to $\adjstrLlCat$ still lands in $\cmdpsCat$.
\end{corollary}

\begin{proof}
We only need to consider 1-cells and 2-cells.
Given a 1-cell $(F,G,\zeta)\colon$ $ (L,R,\eta,\epsilon) \to (L',R',\eta',\epsilon')$,
the diagrams in~(\ref{lem:extadjps}.\ref{lem:extadjps:onecell})
ensure that
$(F,\theta)$ is a lax morphism of comonads
$\adjtocmd(L,R,\eta,\epsilon) \to \adjtocmd(L',R',\eta',\epsilon')$,
where $\theta := L'\zeta^\# \circ \zeta^{-1}R$.
Functoriality follows from~(\ref{rem:adjmorph}.\ref{rem:adjmorph:matecmp}).

It is also clear that $(F,\theta)$ is pseudo whenever $(F,G,\zeta)$ is.
However, the image of a strict left loose morphism $(F,G,\zeta)$ is strict
if and only if $(F,G,\zeta)$ is in fact a strict left morphism.

Given a 2-cell $(\phi,\psi)\colon(F_1,G_1,\zeta_1) \to (F_2,G_2,\zeta_2)$,
we have
\[
L'R'\phi \circ L'\zeta_1^\# \circ \zeta_1^{-1}R
= L'\zeta_2^\# \circ L'\psi R \circ \zeta_1^{-1}R
= \L'\zeta_2^\# \circ \zeta_2^{-1}R \circ \phi LR.
\]
by~(\ref{rem:adjmorph}.\ref{rem:adjmorph:twocell}).
It follows that $\phi$ is a 2-cell
$\jdttpstowccmd(F_1,G_1,\zeta_1) \to \jdttpstowccmd(F_2,G_2,\zeta_2)$.
\end{proof}

\begin{remark}\label{rem:liftunit}
Consider a left loose morphism of adjunctions
$(F,G,\zeta) \colon (L,R,\eta,\epsilon) \to (L',R',\eta',\epsilon')$.
Then (\ref{lem:extadjps}.\ref{lem:extadjps:unitpsnat}) entails that
the natural iso $\zeta \colon L'G \natiso FL$
lifts to a natural iso
\[
\hat{\zeta} \colon \emcmp{L',R'} \circ G \natiso \coal(\adjlootocmd(F,G,\zeta)) \circ \emcmp{L,R}
\]
meaning that $\emfrg{L'R'}\hat{\zeta} = \zeta$.
\end{remark}

\begin{theorem}\label{thm:cmd-adjloo}
The 2-reflection from \ref{thm:cmd-adj} extends along
\hbox{$\adjLCat \hookrightarrow \adjLlCat$} to a bireflection
\[\begin{tikzcd}[column sep=6em]
\cmdCat	\ar[r,shift right=1ex,hook,"\cmdtoadjloo"{swap}] \ar[r,phantom,shift left=.3ex,"\adj{-90}"{description}]
&	\adjLlCat	\ar[l,shift right=1.5ex,"\adjlootocmd"{swap}]
\end{tikzcd}\]
such that the counit is the identity
$\adjlootocmd \circ \cmdtoadjloo = \Id_{\cmdCat}$.
In particular,
the right adjoint $\cmdtoadjloo$ is injective on objects and fully faithful.

The biadjunction restricts between the wide 2-full sub-2-categories on pseudo morphisms.
\end{theorem}

\begin{proof}
It only remains to show that the unit
$\unitadjcmd \colon \Id \Rightarrow \cmdtoadjloo\circ\adjlootocmd$
lifts to a pseudo-natural transformation
$\unitadjcmd \colon \Id \Rightarrow \cmdtoadjloo \circ \adjlootocmd$.
This amounts to give,
for every left loose morphism of adjunctions
$(F,G,\zeta) \colon (L,R,\eta,\epsilon) \to(L',R',\eta',\epsilon')$,
an invertible 2-cell
$(F,\emcmp{L',R'}\circ G,\zeta) \to (F,\coal(\adjlootocmd(F,G,\zeta))\circ\emcmp{L,R},\id)$
in $\adjLlCat$.
For this 2-cell we can take $(\id_F,\hat{\zeta})$,
where $\hat{\zeta}$ is the natural iso from \ref{rem:liftunit}.
\end{proof}

\section{The 2-categories of interest}
\label{sec:defs}

All 2-categories that we shall define below will contain Grothendieck fibrations.

\begin{definition}
\label{def:fibcat}
The 2-category of fibrations $\fibCat$ is the 2-full sub-2-category
of the 2-category of arrows $\Cat^2$
on those functors which are fibrations,
and those morphisms of functors
\[\begin{tikzcd}[ampersand replacement=\&]
	\ee \& {\ee'} \\
	\bb \& {\bb',}
	\arrow["p"', from=1-1, to=2-1]
	\arrow["{p' }", from=1-2, to=2-2]
	\arrow["E", from=1-1, to=1-2]
	\arrow["B"', from=2-1, to=2-2]
\end{tikzcd}\]
such that $E$ maps cartesian arrows to cartesian arrows.
\end{definition}

The 2-cells in $\fibCat$ are the same of $\Cat^2$:
pairs of natural transformations $(\psi,\phi)$
with $\psi \colon B_1 \Rightarrow B_2$ and $\phi \colon E_1 \Rightarrow E_2$
such that $p' \phi = \psi p $.

\subsection{The 2-category of comprehension categories}

\begin{definition}[{\cite[Theorem~9.3.4]{jacobs1999categorical}}]
\label{def:compcat}
A \define{comprehension category (without terminal object)} consists of a fibration $p$
and a morphism $\chi$ of functors over $\ctg{B}$ as depicted below
\[\begin{tikzcd}
\ctg{E}	\ar[dr,"p"'] \ar[rr,"\chi"]
&&	\ctg{B}^2	\ar[dl,"\cod_{\ctg{B}}"]
\\
&	\ctg{B}	&
\end{tikzcd}\]
such that $\chi$ preserves cartesian arrows,
that is, it maps them to pullback squares in $\ctg{B}$.

When $\chi$ is full and faithful, the comprehension category is called \define{full}.
\end{definition}

Comprehension categories are usually required to have terminal objects in $\ctg{B}$.
Here we dispense with this assumption.
Note however that, in all our constructions, the fibration $p$ remains fixed,
and so does its base $\ctg{B}$.

Examples of comprehension categories abound in the literature.
Several of them can be found in~\cite{comprehensioncats,jacobs1999categorical}.
Here we only mention three classes of examples.
Lawvere's \emph{hyperdoctrines with comprehension}~\cite{LawvereF:equhcs};
the \emph{fibration of presheaves} over $\Cat$ with comprehension given by the Grothendieck construction~\cite{ehrhard};
categories $\ctg{C}$ equipped with a class of morphisms $\mathcal{D}$ closed under composition and under pullback along any arrow,
such as fibrations of subobjects, or Brown's \emph{categories with fibrant objects}~\cite{Brown1973}:
the comprehension exhibits $\mathcal{D}$ as the full subfibration of
$\cod \colon \ctg{C}^2 \to \ctg{C}$ on the arrows in $\mathcal{D}$.
A variation on the last example, given a topos, consists in taking the fibration of predicates, \ie arrows into the subobject classifier $\Omega$, instead of subobjects:
the comprehension of a predicate is the subobject it classifies.
Note that the resulting comprehension category is not full~\cite{comprehensioncats}.

\begin{definition}\label{cmpct-morph}
	Let $(p,\chi)$ and $(p',\chi')$ be comprehension categories.
	A \define{lax morphism of comprehension categories} from $(p,\chi)$ to $(p',\chi')$
	is a triple $(B,H,\zeta)$ as in the diagram below,
	such that
	\begin{enumerate}
		\item
		$(B,H)$ is a 1-cell in $\fibCat$, and
		\item
		$\cod \zeta = \ctgid{Bp}$.
	\end{enumerate}
	\[\begin{tikzcd}
		& {\mathcal{B}\due} &&& {\mathcal{B}^{'\mathsf{2}}} \\
		{\mathcal{E}} &&& {\mathcal{E}'} \\
		{\mathcal{B}} &&& {\mathcal{B}'}
		\arrow["p"', from=2-1, to=3-1]
		\arrow["\cod"{pos=0.6}, curve={height=-12pt}, from=1-2, to=3-1]
		\arrow[""{name=0, anchor=center, inner sep=0}, "\chi", from=2-1, to=1-2]
		\arrow["B", from=3-1, to=3-4]
		\arrow["H", from=2-1, to=2-4, crossing over]
		\arrow["{B\due}", from=1-2, to=1-5]
		\arrow["{p'}"', from=2-4, to=3-4]
		\arrow[""{name=1, anchor=center, inner sep=0}, "{\chi'}", from=2-4, to=1-5]
		\arrow["\cod"{pos=0.6}, curve={height=-12pt}, from=1-5, to=3-4]
		\arrow["\zeta", shorten <=29pt, shorten >=29pt, Rightarrow, from=0, to=1]
	\end{tikzcd}\]

	A lax morphism of comprehension categories $(B,H,\zeta)$ is
	a \define{pseudo} (respectively, \define{strict}) \define{morphism of comprehension categories}
	if $\zeta$ is invertible (respectively, the identity).
	
	Given two composable lax morphisms of comprehension categories $(B_1,H_1,\zeta_1)$ and $(B_2,H_2,\zeta_2)$,
	their composite is $(B_2B_1,H_2H_1,(\zeta_2H_1)(B_2^2\zeta_1))$.
	It is straightforward to verify that this composition is unital and associative.
	Pseudo and strict morphisms are clearly closed under composition.
\end{definition}

\begin{example}
Strict morphisms of comprehension categories are considered in~\cite{comprehensioncats,BlancoJ}.
Pseudo morphisms of comprehension categories are considered in~\cite{larrea2018}
\end{example}

\begin{remark}\label{cmpct-morph-triang}
The component at an object $E$ of the natural transformation
$\zeta \colon B\due \chi \Rightarrow \chi' H$
in a lax morphism of comprehension categories
consists of just one arrow,
making the triangle below commute.
\[\begin{tikzcd}
	BX_E	\ar[dr,"{B\chi_E}"{swap}] \ar[rr,"\zeta_E"]
	&&	X'_{HE}	\ar[dl,"\chi'_{HE}"]
	\\
	&	BpE = p'HE	&
\end{tikzcd}\]
\end{remark}

\begin{definition}\label{cmpct-2cats}
	The \define{2-category $\compcatCat$ of comprehension categories} is defined as follows.

	A \define{0-cell} is a comprehension category $(p,\chi)$.

	A \define{1-cell} $(p,\chi) \to (p',\chi')$
	is a lax morphism of comprehension categories~\eqref{cmpct-morph} from $(p,\chi)$ to $(p',\chi')$.

	A \define{2-cell} $(B_1,H_1,\zeta_1)\Rightarrow (B_2,H_2,\zeta_2)$
	is a 2-cell $(\psi,\phi) \colon (B_1,H_1) \Rightarrow (B_2,H_2)$ in $\fibCat$
	as in the left-hand diagram below,
	such that
	$\chi' \phi \circ \zeta_{1} = \zeta_{2} \circ \psi\due\chi$.
	Pasting diagrams, we see that 2-cells are closed under component-wise composition.
	\[\begin{tikzcd}[column sep=4em,row sep=3em]
		{\mathcal{E}} & {\mathcal{E}'} \\
		{\mathcal{B}} & {\mathcal{B}'}
		\arrow["p"', from=1-1, to=2-1]
		\arrow[""{name=0, anchor=center, inner sep=0}, "{B_1}"{description}, curve={height=-12pt}, from=2-1, to=2-2]
		\arrow[""{name=1, anchor=center, inner sep=0}, "{B_2}"{description}, curve={height=12pt}, from=2-1, to=2-2]
		\arrow[""{name=2, anchor=center, inner sep=0}, "{H_2}"{description}, curve={height=12pt}, from=1-1, to=1-2]
		\arrow[""{name=3, anchor=center, inner sep=0}, "{H_1}"{description}, curve={height=-12pt}, from=1-1, to=1-2]
		\arrow["{p'}", from=1-2, to=2-2]
		\arrow["\psi", shorten <=5pt, shorten >=5pt, Rightarrow, from=0, to=1]
		\arrow["\phi", shorten <=5pt, shorten >=5pt, Rightarrow, from=3, to=2]
	\end{tikzcd}
	\hspace{3em}
	\begin{tikzcd}[sep=3em]
		{\mathcal{B}\due} & {\mathcal{B}^{'\mathsf{2}}} \\
		{\mathcal{E}} & {\mathcal{E}'}
		\arrow["\chi", from=2-1, to=1-1]
		\arrow[""{name=0, anchor=center, inner sep=0}, "{H_1}"{description}, curve={height=-12pt}, from=2-1, to=2-2]
		\arrow[""{name=1, anchor=center, inner sep=0}, "{H_2}"{description}, curve={height=12pt}, from=2-1, to=2-2]
		\arrow[""{name=2, anchor=center, inner sep=0}, "{B_1\due}"{description}, curve={height=-12pt}, from=1-1, to=1-2]
		\arrow["{\chi'}"', from=2-2, to=1-2]
		\arrow["\phi", shorten <=5pt, shorten >=5pt, Rightarrow, from=0, to=1]
		\arrow["{\zeta_1}", shorten <=9pt, shorten >=9pt, Rightarrow, from=2, to=0]
	\end{tikzcd}
	\hspace{3em}
	\begin{tikzcd}[sep=3em]
		{\mathcal{B}\due} & {\mathcal{B}^{'\mathsf{2}}} \\
		{\mathcal{E}} & {\mathcal{E}'}
		\arrow["\chi", from=2-1, to=1-1]
		\arrow[""{name=0, anchor=center, inner sep=0}, "{H_2}"{description}, curve={height=12pt}, from=2-1, to=2-2]
		\arrow[""{name=1, anchor=center, inner sep=0}, "{B_1\due}"{description}, curve={height=-12pt}, from=1-1, to=1-2]
		\arrow["{\chi'}"', from=2-2, to=1-2]
		\arrow[""{name=2, anchor=center, inner sep=0}, "{B_2\due}"{description}, curve={height=12pt}, from=1-1, to=1-2]
		\arrow["{\zeta_2}", shorten <=9pt, shorten >=9pt, Rightarrow, from=2, to=0]
		\arrow["{\psi\due}", shorten <=5pt, shorten >=5pt, Rightarrow, from=1, to=2]
	\end{tikzcd}\]
	We write $\compcatpsCat$ and $\compcatstrCat$
	for the 2-full 2-subcategory of $\compcatCat$ with the same 0-cells
	and only those 1-cells which are pseudo (respectively, strict) morphisms of comprehension categories.
\end{definition}

It is straightforward to verify that the composition of lax morphisms of comprehension categories is unital and associative,
as it is that of 2-cells.

\begin{remark}\label{cmpcat-2cell-sq}
	Let $(B_1,H_1,\zeta^1),(B_2,H_2,\zeta^2)\colon(p,\chi)\to(p',\chi')$
	be lax morphisms of comprehension categories.
	A 2-cell
	$(\psi,\phi)\colon(B_1,H_1)\Rightarrow(B_2,H_2)$ in $\fibCat$
	is a 2-cell in $\compcatCat$ if and only if, for every $E$ in $\ctg{E}$ over $X$,
	the top square in the diagram below commutes,
	\[\begin{tikzcd}[row sep=1.5em]
		{B_1X_E} && {B_2X_E} \\
		& {X'_{H_1E}} && {X'_{H_2E}} \\
		{B_1X} && {B_2X}
		\arrow["{\chi'_{H_1E}}"{pos=0.3}, from=2-2, to=3-1]
		\arrow["{\zeta^1_E}", from=1-1, to=2-2]
		\arrow["{B_1\chi_E}"'{pos=0.8}, from=1-1, to=3-1]
		\arrow["{\psi_X}"', from=3-1, to=3-3]
		\arrow["{\psi_{X_E}}", from=1-1, to=1-3]
		\arrow["{\chi'_{H_2E}}"{pos=0.3}, from=2-4, to=3-3]
		\arrow["{\zeta^2_E}", from=1-3, to=2-4]
		\arrow["{B_2\chi_E}"'{pos=0.8}, from=1-3, to=3-3]
		\arrow[from=2-2, to=2-4, crossing over]
	\end{tikzcd}\]
	where the front square is the image under $\chi'$ of $\phi_E \colon H_1E \to H_2E$,
	the back square is naturality of $\psi$, and
	the side triangles are those from \eqref{cmpct-morph-triang}.

	If $(B_1,H_1,\zeta^1)$ and $(B_2,H_2,\zeta^2)$ are strict morphisms,
	the top square above commutes if and only if
	its horizontal arrows coincide.
	Therefore $(\psi,\phi)$ is a 2-cell between strict morphisms if and only if $\dom\chi'\phi = \psi\dom\chi$.
\end{remark}

\subsection{Weakening and contraction comonads}

Here we recall the intermediate notion,
the weakening and contraction comonads introduced by Jacobs,
that we use to compare comprehension categories and \gencwf{ies}.

\begin{definition}[{\cite[Def. 9.3.1]{jacobs1999categorical}}]\label{wccmd-def}
	Let $p \colon \ctg{E} \to \ctg{B}$ be a fibration.
	A \define{weakening and contraction comonad on $p$}, or \define{\wccmd} for short,
	is a comonad $(K,\epsilon,\nu)$ on $\ctg{E}$ such that
	\begin{enumerate}
	\item\label{wccmd-def:car}
	the counit $\epsilon$ is $p$-cartesian and,
	\item\label{wccmd-def:pb}
	for every cartesian arrow $f \colon A \car B$ in $\ctg{E}$
	the image in $\ctg{B}$ under $p$
	\[\begin{tikzcd}
			pKA	\ar[d,"pKf"{swap}] \ar[r,"{p\epsilon_A}"]
			&	pA	\ar[d,"pf"]
			\\
			pKB	\ar[r,"{p\epsilon_B}"]	&	pB
	\end{tikzcd}\]
	of the naturality square of $\epsilon$ is a pullback square.
	\end{enumerate}
	We may write $pA.A$ for $pKA$,
	and we may say \wccmd{} to mean the pair of a fibration and a \wccmd{} on it.
\end{definition}

\begin{remark}\label{wccmd-rmk}
\begin{enumerate}
\item\label{wccmd-rmk:pb'}
For every cartesian arrow $f$,
the naturality square of the counit $\epsilon$
\[\begin{tikzcd}
KA	\ar[d,"Kf"{swap}] \ar[r,"{\epsilon_A}"]
&	A	\ar[d,"f"]
\\
KB	\ar[r,"{\epsilon_B}"]	&	B
\end{tikzcd}\]
is a pullback.
This follows from the fact that, in general,
a square in $\ctg{E}$ is a pullback if
it has two parallel cartesian sides
and it is sitting over a pullback in $\ctg{B}$.
\item\label{wccmd-rmk:def}
Given a \wccmd{} $(K,\epsilon,\nu)$ on a fibration $p$,
the naturality square of $\epsilon$ for $\epsilon_A$ itself is a pullback.
It follows that the comultiplication is canonically determined by the counit $\epsilon$ via the two counitality axioms.
Thus one could equivalently define a \wccmd{} to be a copointed endofunctor which enjoys
conditions~\ref{wccmd-def:car} and~\ref{wccmd-def:pb} in \eqref{wccmd-def}.
See also~\cite[p.536]{jacobs1999categorical}.
It also follows that coalgebras for the copointed endofunctor coincide with coalgebras for the comonad.
\end{enumerate}
\end{remark}

\begin{remark}\label{wccmd-coalg}
Given any fibration, 
if a composite $gf$ is cartesian and $g$ is cartesian,
then $f$ is cartesian too.
Two immediate consequences of the fact
that the counit of a \wccmd{} is cartesian are:
\begin{enumerate}
\item
The functor $K$ in a \wccmd{} preserves cartesian arrows.
\item
If $(E,e)$ is a coalgebra for a \wccmd,
then $e$ is a cartesian arrow.
\end{enumerate}
\end{remark}

The 2-category of \wccomonad[s] is a strict 2-pullback over $\Cat$
of the 2-category of fibrations~\eqref{def:fibcat} and the 2-category of comonads~\eqref{def:cmdCat}.

\begin{definition}\label{2catofwccmd}
	The \define{2-category $\wccmdCat$ of \wccomonad[s]} is defined as follows.

	A {\it 0-cell} is a pair $(p,K)$ with $p$ a fibration and $K$ a \wccmd{} on $p$.

	A \define{1-cell} from $(p,K)$ to $(p',K')$
	is a triple $(C,H,\theta)$ as in the diagram below, such that

	\begin{minipage}{.65\textwidth}
		\begin{enumerate}
			\item
			$(C,H)\colon p\to p'$ is a 1-cell in $\fibCat$
			\item
			$(H,\theta)\colon K\to K'$ is a 1-cell in $\cmdCat$.
		\end{enumerate}
	\end{minipage}
	\begin{minipage}{.3\textwidth}
		\[\begin{tikzcd}
			{\mathcal{E}} & {\mathcal{E}'} \\
			{\mathcal{E}} & {\mathcal{E}'} \\
			{\mathcal{C}} & {\mathcal{C}'}
			\arrow[""{name=0, anchor=center, inner sep=0}, "H"{description}, from=1-1, to=1-2]
			\arrow[""{name=1, anchor=center, inner sep=0}, "H"{description}, from=2-1, to=2-2]
			\arrow["C", from=3-1, to=3-2]
			\arrow["K", from=2-1, to=1-1]
			\arrow["{K'}"', from=2-2, to=1-2]
			\arrow["p"', from=2-1, to=3-1]
			\arrow["{p'}", from=2-2, to=3-2]
			\arrow["\theta", shorten <=6pt, shorten >=6pt, Rightarrow, from=0, to=1]
		\end{tikzcd}\]
	\end{minipage}\\

	A \define{2-cell} from $(C_1,H_1,\theta_1)$ to $(C_2,H_2,\theta_2)$
	is a 2-cell $(\psi,\phi) \colon (C_1,H_1) \to (C_2,H_2)$ in $\fibCat$
	as in the left-hand diagram below,
	such that $\phi$ is a 2-cell $(H_1,\theta_1) \to (H_2,\theta_2)$ in $\cmdCat$,
	as in the right-hand side.
	\[\begin{tikzcd}[column sep=3em,row sep=3em]
		\mathcal{E}
		& \mathcal{E'}
		\\
		{\mathcal{C}} & {\mathcal{C}'}
		\arrow["p"', from=1-1, to=2-1]
		\arrow["{p'}", from=1-2, to=2-2]
		\arrow[""{name=4, anchor=center, inner sep=0}, "{C_1}"{description}, curve={height=-12pt}, from=2-1, to=2-2]
		\arrow[""{name=6, anchor=center, inner sep=0}, "{C_2}"{description}, curve={height=12pt}, from=2-1, to=2-2]
		\arrow[""{name=0, anchor=center, inner sep=0}, "{H_1}"{description}, curve={height=-12pt}, from=1-1, to=1-2]
		\arrow[""{name=2, anchor=center, inner sep=0}, "{H_2}"{description}, curve={height=12pt}, from=1-1, to=1-2]
		\arrow["\phi", shorten <=3pt, shorten >=3pt, Rightarrow, from=0, to=2]
		\arrow["\psi", shorten <=3pt, shorten >=3pt, Rightarrow, from=4, to=6]
	\end{tikzcd}
	\hspace{3em}
	\begin{tikzcd}[column sep=3em,row sep=3em]
		{\mathcal{E}} & {\mathcal{E}'} & {\mathcal{E}} & {\mathcal{E}'} \\
		{\mathcal{E}} & {\mathcal{E}'} & {\mathcal{E}} & {\mathcal{E}'}
		\arrow["{K'}"'{name=lh}, from=2-2, to=1-2]
		\arrow["K", from=2-1, to=1-1]
		\arrow[""{name=0, anchor=center, inner sep=0}, "{H_1}"{description}, curve={height=-12pt}, from=1-1, to=1-2]
		\arrow[""{name=1, anchor=center, inner sep=0}, "{H_2}"{description}, curve={height=12pt}, from=2-1, to=2-2]
		\arrow["{K'}"', from=2-4, to=1-4]
		\arrow["K"{name=rh}, from=2-3, to=1-3]
		\arrow[""{name=2, anchor=center, inner sep=0}, "{H_1}"{description}, curve={height=-12pt}, from=1-3, to=1-4]
		\arrow[""{name=3, anchor=center, inner sep=0}, "{H_2}"{description}, curve={height=12pt}, from=2-3, to=2-4]
		\arrow[""{name=4, anchor=center, inner sep=0}, "{H_1}"{description}, curve={height=-12pt}, from=2-1, to=2-2]
		\arrow[""{name=5, anchor=center, inner sep=0}, "{H_2}"{description}, curve={height=12pt}, from=1-3, to=1-4]
		\arrow["{\theta_1}", shorten <=4pt, shorten >=4pt, Rightarrow, from=0, to=4]
		\arrow["\phi", shorten <=3pt, shorten >=3pt, Rightarrow, from=4, to=1]
		\arrow["\phi", shorten <=3pt, shorten >=3pt, Rightarrow, from=2, to=5]
		\arrow["{\theta_2}", shorten <=4pt, shorten >=4pt, Rightarrow, from=5, to=3]
		\arrow[from=lh,to=rh,phantom,"="{description}]
	\end{tikzcd}\]

	We write $\wccmdpsCat$ and $\wccmdstrCat$ for the 2-full 2-subcategories of $\wccmdCat$ with the same 0-cells,
	and only those 1-cells $(C,H,\theta)$ such that $(H,\theta)$ is a pseudo (respectively, strict) morphism of comonads.
\end{definition}

\subsection{The 2-category of \gencwf{ies}}

\begin{definition}[{\cite[Def. 3.0.1]{cjd}}]
\label{jdtt}
A \define{\gencwf{y}}, \define{\gcwf{}} for short,
is the data of a morphism $\Sigma$ of fibrations over the same base $\ctg{B}$
as depicted below,
together with a right adjoint $\Delta$ to $\Sigma$
such that the components of both unit and counit are cartesian
with respect to $\du$ and $u$, respectively.
\[\begin{tikzcd}
	{\duu} && {\uu} \\
	& {\ctg{B}}
	\arrow["{\dot{u}}"', from=1-1, to=2-2]
	\arrow["u", from=1-3, to=2-2]
	\arrow[""{name=0, anchor=center, inner sep=0}, "\Sigma"', curve={height=10pt}, from=1-1, to=1-3]
	\arrow[""{name=1, anchor=center, inner sep=0}, "\Delta"', curve={height=10pt}, from=1-3, to=1-1]
	\arrow["\dashv"{anchor=center, rotate=90}, draw=none, from=0, to=1]
\end{tikzcd}\]
\end{definition}

Notice that the adjunction $\Sigma\dashv\Delta$ in \emph{not} fibred:
the triangle involving $\Delta$ does not commute,
\ie $\Delta$ is not a morphism of functors,
and the unit and counit are cartesian rather then vertical.
Still, in~\eqref{deltacartesian} we will show that it inherits some desirable fibrational properties.

Of course, a category with families~\cite{dybjer_inttt} is the same thing as a \gencwf{y} with discrete fibrations $u$ and $\du$ and a terminal object in $\ctg{B}$, as implied in the following example.

\begin{example}[The free syntactic (g)cwf]\label{exa:free_syntactic}
Given a calculus of dependent types à la Martin-L\"of~\cite{mltt},
see~\cite{Rijke_intro} for an introduction,
one can build a (generalised) category with families as follows,
see \eg~\cite[\S5.5]{PALMGREN2019102715} for the proofs and more details.

\[\begin{tikzcd}[row sep=1ex,column sep=5ex]
\hspace{2ex}{\small\text{$\Gamma.A \vdash \mathtt{v}_A : A$}}
&&
{\small\text{$\Gamma \vdash A\Ty$}}\hspace{4.5ex}	\ar[ll,mapsto]
\\
{\duu = \{{\small \text{$\Gamma \vdash a : A$}}\}}
\ar[dddr,start anchor={[xshift=2ex]south west},bend right=5ex,"\du"']
\ar[rr,shift right=1.1ex,"{\mathrm{T}}"',""{name=la}]
&&
{\{{\small \text{$\Gamma \vdash A\Ty$}}\} = \uu}
\ar[dddl,start anchor={[xshift=-2ex]south east},bend left=5ex,"u"]
\ar[ll,shift right=1.4ex,"{\mathrm{V}}"',""{name=ra}]
\\
\hspace{4.5ex}{\small \text{$\Gamma \vdash a : A$}}	\ar[rr,mapsto] \ar[dr,mapsto,bend right=3ex]
&&
{\small \text{$\Gamma \vdash A\Ty$}}\hspace{4.5ex}	\ar[dl,mapsto,bend left=3ex]
\\[5ex]
&	\Gamma
&\\
&	\ctxcat	&\ar[from=la,to=ra,phantom,"\adj{90}"{description}]
\end{tikzcd}\]

First of all, we can define a category $\ctxcat$ of contexts, whose objects are (equivalence classes of definitionally equal) well-formed contexts of the form $\Gamma=x_1:A_1, \dots, x_n:A_n$ and whose morphisms are (equivalence classes of definitionally equal) terms
\[
t=(t_1,\dots,t_n)\colon\Theta\to\Gamma
\]
where $\Theta\vdash t_1:A_1$ and $\Theta\vdash t_i:A_i[t_1/x_1,\dots,t_{i-1}/x_{i-1}]$ for $i=2:n$. We ought to think of these as substitutions from $\Theta$ into $\Gamma$, with composition being iterated substitution and identity the trivial $(x_1,\dots,x_n)\colon\Gamma\to\Gamma$. The empty context is the terminal object in $\ctxcat$ so defined.

In what follows, in order to improve readability, we omit repeating that all contexts, types, and terms are intended `up to definitional equality', but it is so throughout this construction.
Now, the category of types $\uu$ is that of type judgements and type substitutions: mapping each type judgement to its context provides the structure of a discrete fibration $u\colon\uu\to\ctxcat$. The fibre over each $\Gamma$, then, is the \emph{set} $\uu_\Gamma = \{ \Gamma\vdash A\Ty \}$ which is precisely the image of $\Gamma$ through the presheaf of types $\mathrm{Ty}\colon\ctxcat\opp\to\catof{Set}$ in the classical definition of a cwf, and reindexing along a context morphism $t\colon\Theta\to\Gamma$ precisely computes substitution in types, $\Gamma\vdash A\Ty\mapsto\Theta\vdash A[t/x]\Ty$. Similarly we can define a discrete fibration $\du\colon\duu\to\ctxcat$ classifying terms: its total category is that of typing judgements and term substitutions, which are mapped, respectively, to their underlying context and context morphism.

On top of that we can define an adjoint pair $\mathrm{T}\dashv\mathrm{V}$ where the two functors act as in the two following rules involving the structure of judgements\footnote{While $\mathrm{V}$ is a proper `structural' rule, and it is often assumed, $\mathrm{T}$ describes the structure but is usually derivable in a given calculus of dependent types.}.
\[
\frac{\Gamma\vdash a:A}{\Gamma \vdash A\Ty}(\mathrm{T})
\hspace{4em}
\frac{\Gamma\vdash A\Ty}{\Gamma.A \vdash \texttt{v}_A : A}(\mathrm{V})
\]
Notice that $\mathrm{T}$ makes the obvious triangle commute because contexts are preserved and a morphism of typing judgements is in particular a morphism of type judgements.\footnote{Writing this back to back we realise the way we call the two might lead to some confusion, but we hope to make the distinction clear along the way.}
Being both fibrations discrete, $\mathrm{T}$ is trivially cartesian: notice that this implies \cite[Lemma 2.1]{streicher2022fibred} that it is itself a fibration, so that terms are fibred over types, as well. On the other hand, as we said just before this discussion, notice that it is key that $\mathrm{V}$ does \emph{not} add up to a functor morphism $u\to\du$, or the context would be preserved and we would \emph{not} have context extension.

Finally, we unpack the unit and counit needed: again, they will be cartesian `for free', since both fibrations are discrete. We begin with the counit, whose components need to be morphisms of type judgements $\Gamma.A \vdash A\Ty\to\Gamma\vdash A\Ty$: one can show that the cartesian lifting (\ie substitution) of $(x_1,\dots,x_n)\colon\Gamma.A\to\Gamma$ at (\ie in) $A$ has the desired universal property: it does basically nothing, as expected by weakening. This is often called \emph{projection} or \emph{display}, depending on which model one is considering. The unit, instead, has less popular correspondents in the literature, and at a term $\Gamma\dashv a:A$ it is the cartesian lifting of $(x_1,\dots,x_n,a)\colon\Gamma\to\Gamma.A$ at $\texttt{v}_A$ -- that is substituting $a$ into the fresh free variable produced by context extension.
\end{example}

Examples of \gencwf{ies} are described in \cite[\S\S3-5]{cjd}: among others, they arise from categories with finite products, from Lawvere-style doctrines, from topoi.

\begin{remark}
Since the free syntactic object produced out of a calculus of dependent types produces fibrations that are discrete, one could wonder whether from a type-theoretic perspective it might only be worth to give an account of the discrete case. Elsewhere \cite{ce_subtyping}, we argue that the `remaining' vertical portion of a \gcwf{} is actually apt to describe dependent types \emph{with subtyping}.
\end{remark}

Next, we make few simple observations on \gencwf{ies}.

\begin{remark}\label{jdtt-unitmono}
Each component of the unit in a \gcwf{} is a monic arrow.
Indeed, let $f,g \colon a \to b$ in $\duu$ be such that $\eta_b f = \eta_b g$.
It follows that
\[
\du f = (u\epsilon_{\Sigma b})(\du \eta_b)(\du f)
= (u\epsilon_{\Sigma b})(\du \eta_b)(\du g) = \du g
\]
and, in turn, that $f = g$ since $\eta_b$ is cartesian.
\end{remark}

\begin{lemma}\label{jdtt-leftfaith-iso}
Let $(u,\du,\Sigma\dashv\Delta)$ be a \gcwf.
The left adjoint $\Sigma$ induces a bijection
\[\begin{tikzcd}[column sep=4em]
\duu(a,b)	\ar[r,"\sim"]
&	\{ f \in \uu(\Sigma a, \Sigma b) \mid (\Sigma\eta_b)f = (\Sigma\Delta f)(\Sigma\eta_a) \}.
\end{tikzcd}\]
\end{lemma}

\begin{proof}
The counter-image of $f$ in the right-hand set is the (only) arrow $g$ in $\duu(a,b)$ over $uf$
such that $\eta_b g = (\Delta f) \eta_a$.
It exists since $\eta_b$ is cartesian.

To see that $\Sigma$ is faithful, use~\eqref{jdtt-unitmono}
and the naturality square of the unit.
\end{proof}

The next lemma shows that $\Delta$ is a cartesian functor.
Still, recall that $\Delta$ is not (required to be) a morphism of functors over $\ctg{B}$.

\begin{lemma}\label{deltacartesian}
Let $(u,\du,\Sigma\dashv\Delta)$ be a \gcwf.
Then we have that
\begin{enumerate}
\item\label{deltacartesian-iff} $\Delta$ preserves cartesian maps iff $\Sigma$ reflects cartesian maps;
\item\label{deltacartesian-car} $\Delta$ preserves cartesian maps.
\item\label{deltacartesian-refl} $\Sigma$ reflects cartesian maps.
\end{enumerate}
\end{lemma}
\begin{proof}
Let us start with~\eqref{deltacartesian-iff}.
From left to right,
let $f\colon a\to b$ in $\duu$ such that $\Sigma f$ is cartesian,
then $\Delta\Sigma f$ is cartesian, and we have the following
\[\begin{tikzcd}[ampersand replacement=\&]
	a \& b \&\& {\Sigma a} \& {\Sigma b} \\
	{\Delta\Sigma a} \& {\Delta\Sigma b}
	\arrow["f", from=1-1, to=1-2]
	\arrow["{\Delta\Sigma f}"', from=2-1, to=2-2]
	\arrow["{\eta_a}"', from=1-1, to=2-1]
	\arrow["{\eta_b}", from=1-2, to=2-2]
	\arrow["{\Sigma f}", from=1-4, to=1-5]
\end{tikzcd}\]
with cartesian units, hence $\eta_b f$ is cartesian with $\eta_b$ cartesian.
By~\eqref{wccmd-coalg}, $f$ is cartesian too.
The converse can be worked out the dual way using counits.

Next, we prove~\eqref{deltacartesian-car}.
Let $h\colon A \to B$ in $\uu$ be cartesian
and consider $f\colon c\to \Delta B$ and $\phi \colon \du c \to \du \Delta A$ such that $\du f = \du \Delta h \circ \phi$,
as in the left-hand diagrams below.
\[\begin{tikzcd}[ampersand replacement=\&]
	c \&\&\&\& {\Sigma c} \\
	\& {\Delta A} \& {\Delta B} \&\&\& {\Sigma\Delta A} \& {\Sigma\Delta B} \\
	\&\&\&\& A \& B \\
	\&\& {\du c} \\
	\&\&\& {\du\Delta A} \& {\du \Delta B} \\
	\&\& {u A} \& uB
	\arrow["h"', from=3-5, to=3-6]
	\arrow["{\Delta h}"', from=2-2, to=2-3]
	\arrow["f", curve={height=-6pt}, from=1-1, to=2-3]
	\arrow["{\du \Delta h}"'{pos=0.2}, from=5-4, to=5-5]
	\arrow["uh"', from=6-3, to=6-4]
	\arrow["{u\epsilon_B}", from=5-5, to=6-4]
	\arrow["{u\epsilon_A}"', from=5-4, to=6-3]
	\arrow["{\du f}", curve={height=-6pt}, from=4-3, to=5-5]
	\arrow["\phi"', from=4-3, to=5-4]
	\arrow["{g'}"', dashed, from=1-1, to=2-2]
	\arrow["{\Sigma\Delta h}"', from=2-6, to=2-7]
	\arrow["{\epsilon_A}"', from=2-6, to=3-5]
	\arrow["{\epsilon_B}", from=2-7, to=3-6]
	\arrow["{\Sigma f}", curve={height=-6pt}, from=1-5, to=2-7]
	\arrow["g"', dotted, from=1-5, to=2-6]
\end{tikzcd}\]
Note first that $\Sigma\Delta h$ is cartesian by~\eqref{wccmd-coalg} because its postcomposition with $\epsilon_B$ is.
It follows that there exists a unique dotted map
$g\colon \Sigma c\to \Sigma \Delta A$
that post-composed with $\Sigma\Delta h$ is $\Sigma f$.
We can then take the transpose of the composite $\epsilon_A g$ to be $g'$.
The left-hand triangle commutes since the right-hand diagram does.
This $g'$ the unique such since, in addition, 
$\Sigma$ is faithful by~\eqref{jdtt-leftfaith-iso}.
\end{proof}

Morphisms of \gencwf{ies} are defined using morphisms of fibrations~\eqref{def:fibcat} and morphisms of adjunctions~\eqref{def:adjmorph}.

\begin{definition}\label{jdtt-morph}
	Let $\jdtt{U}=(u,\du,\Sigma\dashv\Delta)$ and $\jdtt{U}'=(u',\du',\Sigma'\dashv\Delta')$ be \gcwfs.
	A \define{(lax) loose \gcwf{} morphism} from $\jdtt{U}$ to $\jdtt{U}'$
		is a quadruple $(C,H,\dot{H},\zeta)$ such that
		\begin{enumerate}
			\item
			$(C,H)\colon u\to u'$ is a 1-cell in $\fibCat$,
			\item
			$(C,\dot{H})\colon \dot{u}\to\dot{u}'$ is a 1-cell in $\fibCat$, and
			\item
			$(H,\dot{H},\zeta)\colon(\Sigma,\Delta)\to(\Sigma',\Delta')$ is a 1-cell in $\adjLlCat$\!\!, \ie a left loose morphism of adjunctions.
			In particular $\zeta \colon \Sigma'\dot{H}\natiso H\Sigma $.
		\end{enumerate}
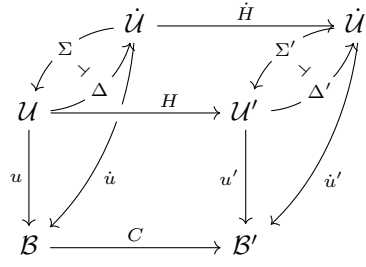
	\[\begin{tikzcd}
		& \duu && {\duu'} \\
		\uu && {\uu'} \\
		\\
		\ctg{B} && {\ctg{B}'}
		\arrow["u"', from=2-1, to=4-1]
		\arrow["\du"{pos=.6}, curve={height=-12pt}, from=1-2, to=4-1]
		\arrow["{u'}"', from=2-3, to=4-3]
		\arrow["{\du'}"{pos=.6}, curve={height=-12pt}, from=1-4, to=4-3]
		\arrow["C", from=4-1, to=4-3]
		\arrow["H"{pos=.7}, from=2-1, to=2-3,crossing over]
		\arrow["{\dot{H}}", from=1-2, to=1-4]
		\arrow[""{name=0, anchor=center, inner sep=0}, "\Sigma"{description}, curve={height=12pt}, from=1-2, to=2-1]
		\arrow[""{name=1, anchor=center, inner sep=0}, "\Delta"{description}, curve={height=12pt}, from=2-1, to=1-2]
		\arrow[""{name=2, anchor=center, inner sep=0}, "{\Delta'}"{description}, curve={height=12pt}, from=2-3, to=1-4]
		\arrow[""{name=3, anchor=center, inner sep=0}, "{\Sigma'}"{description}, curve={height=12pt}, from=1-4, to=2-3]
		\arrow["\dashv"{anchor=center, rotate=-43}, draw=none, from=0, to=1]
		\arrow["\dashv"{anchor=center, rotate=-42}, draw=none, from=3, to=2]
	\end{tikzcd}\]

	A \define{(lax) \gcwf{} morphism} from $\jdtt{U}$ to $\jdtt{U}'$ is a loose \gcwf{} morphism $(C,H,\dot{H},\zeta)$
	such that $\zeta = \id \colon \Sigma'\dot{H} = H\Sigma$.

	A loose \gcwf{} morphism $(C,H,\dot{H},\zeta)$ is \define{pseudo} (respectively \define{strict})
	when the corresponding 1-cell in $\adjLlCat$ is.
\end{definition}

The 2-category of \gencwf{ies} is a pullback, over $\Cat \times \Cat \times \Cat$,
involving the 2-category of fibrations~\eqref{def:fibcat} (two times)
and the ``left'' 2-category of adjunctions~\eqref{def:adjCat}.

\begin{definition}\label{jdtt-2cat}
	The \define{2-category $\jdttlCat$ of \gcwfs{} and loose \gcwf{} morphisms}
	has these as objects and arrows, and a 2-cell
	$(C_1,H_1,\dot{H}_1,\zeta_1) \to (C_2,H_2,\dot{H}_2,\zeta_2)$
	is a triple $(\phi,\dot{\phi},\psi)$ of natural transformations as in the left-hand diagram below,
	such that
	\begin{enumerate}
		\item\label{jdtt-2cat:fibty}
		$(\psi,\phi)$ is a 2-cell $(C_1,H_1) \to (C_2,H_2)$ in $\fibCat$ (\ie in $\Cat^2$),
		\item\label{jdtt-2cat:fibt}
		$(\psi,\dot{\phi})$ is a 2-cell $(C_1,\dot{H}_1) \to (C_2,\dot{H}_2)$ in $\fibCat$, and
		\item\label{jdtt-2cat:leftsq}
		$(\phi,\dot{\phi})$ is a 2-cell $(H_1,\dot{H}_1,\zeta_1) \to (H_2,\dot{H}_2,\zeta_2)$ in $\adjLlCat$,
		meaning that the right-hand diagram below commutes.
	\end{enumerate}
	\[\begin{tikzcd}[column sep=2em]
		& \duu &&& {\duu'} \\
		\uu &&& {\uu'} \\
		\\
		\ctg{B} &&& {\ctg{B}'}
		\arrow["u"'{pos=0.6}, from=2-1, to=4-1]
		\arrow["\Sigma"', from=1-2, to=2-1]
		\arrow["\du"{pos=0.6}, curve={height=-12pt}, from=1-2, to=4-1]
		\arrow[""{name=0, anchor=center, inner sep=0}, "{\dot{H}_1}"{description}, curve={height=-12pt}, from=1-2, to=1-5]
		\arrow[""{name=1, anchor=center, inner sep=0}, "{\dot{H}_2}"{description}, curve={height=12pt}, from=1-2, to=1-5]
		\arrow["{\Sigma'}", from=1-5, to=2-4]
		\arrow["{\du'}"{pos=0.6}, curve={height=-12pt}, from=1-5, to=4-4]
		\arrow[""{name=2, anchor=center, inner sep=0}, "{H_1}"{description}, curve={height=-12pt}, from=2-1, to=2-4, crossing over]
		\arrow[""{name=3, anchor=center, inner sep=0}, "{H_2}"{description}, curve={height=12pt}, from=2-1, to=2-4, crossing over]
		\arrow["{u'}"'{pos=0.6}, from=2-4, to=4-4]
		\arrow[""{name=4, anchor=center, inner sep=0}, "{C_1}"{description}, curve={height=-12pt}, from=4-1, to=4-4]
		\arrow[""{name=5, anchor=center, inner sep=0}, "{C_2}"{description}, curve={height=12pt}, from=4-1, to=4-4]
		\arrow["{\dot{\phi}}", shorten <=3pt, shorten >=3pt, Rightarrow, from=0, to=1]
		\arrow["\phi", shorten <=3pt, shorten >=3pt, Rightarrow, from=2, to=3]
		\arrow["\psi", shorten <=3pt, shorten >=3pt, Rightarrow, from=4, to=5]
	\end{tikzcd}
	\hspace{3em}
	\begin{tikzcd}[row sep=2em,column sep=2em]
	\Sigma'\dot{H}_1	\ar[d,Rightarrow,"\Sigma'\psi"'] \ar[r,Rightarrow,"{\zeta_1}"]
	&	H_1\Sigma	\ar[d,Rightarrow,"\phi \Sigma"]
	\\
	\Sigma'\dot{H}_2	\ar[r,Rightarrow,"\zeta_2"']
	&	H_2\Sigma
	\end{tikzcd}
	\]

	The \define{2-category $\jdttCat$ of \gcwfs{} and \gcwf{} morphisms}
	is defined as the wide 2-full sub-2-category of $\jdttlCat$
	on the \gcwf{} morphisms.

	We write $\jdttlpsCat$, $\jdttpsCat$, and $\jdttstrCat$ for the 2-full 2-subcategories of $\jdttlCat$ and $\jdttCat$ on the 1-cells which are pseudo and strict morphisms, respectively.
\end{definition}

\section{The biequivalence between comprehension categories and \gencwf{ies}}
\label{sec:bieqv}

In this section we shall prove the following result.

\begin{theorem}\label{eqv-compc-gcwf}
There is an adjoint biequivalence.
\[\begin{tikzcd}[column sep=6em]
\compcatCat	\ar[r,shift right=6pt,"\comptogcfw"'] \ar[r,phantom,"\equiv"{description}]
&	\jdttlCat	\ar[l,shift right=6pt,"\gcwftocomp"']
\end{tikzcd}\]
such that $\gcwftocomp \circ \comptogcfw = \Id$.

The biequivalence restricts to the wide 2-full sub-2-categories on the pseudo morphisms.
\end{theorem}

The biequivalence is obtained composing the two biequivalences appearing in the diagram below, which commutes appropriately (left and right adjoints separately)
and where the vertical arrows are the obvious forgetful functors.

All top 2-categories have a forgetful to $\fibCat$ (for a \gcwf{} $(u,\du,\Sigma\dashv\Delta)$ it is $u$),
and all top 2-functors commute (strictly) with these forgetful 2-functors.

\[\begin{tikzcd}[column sep=2.5em,row sep=2.5em]
\compcatCat		\ar[rr,shift right=1ex,"\cmpcttowccmd"{swap}]
				\ar[rr,phantom,shift left=.3ex,"\equiv"{description}]
&&	\wccmdCat	\ar[dd]
				\ar[ll,shift right=1.5ex,"\wccmdtocmpct"{swap}]
				\ar[rrr,shift right=1ex,"\wccmdtojdttps"{swap}]
				\ar[rrr,phantom,shift left=.3ex,"\equiv"{description}]
&&&	\jdttlCat	\ar[dd]
				\ar[lll,shift right=1.5ex,"\jdttpstowccmd"{swap}]
\\
&	\wccmdCat	\ar[dd] \ar[ur,"\Id"]
			\ar[rrr,shift right=1ex, crossing over,"\wccmdtojdtt"{swap}]
			\ar[rrr,phantom,shift left=.3ex,"\adj{-90}"{description}]
&&&	\jdttCat	\ar[ur,hook]
				\ar[lll,shift right=1.5ex, crossing over,"\jdtttowccmd"{swap}]
&\\
&&	\cmdCat	\ar[rrr,shift right=1ex,hook,"\cmdtoadjloo"{swap}]
			\ar[rrr,phantom,shift left=.3ex,"\adj{-90}"{description}]
&&&	\adjLlCat	\ar[lll,shift right=1.5ex,"\adjlootocmd"{swap}]
				\ar[from=2-5,to=4-5,crossing over]
\\
&	\cmdCat	\ar[ur,"\Id"]
		\ar[rrr,shift right=1ex,hook,"\cmdtoadj"{swap}]
		\ar[rrr,phantom,shift left=.3ex,"\adj{-90}"{description}]
&&&	\adjLCat	\ar[ur,hook]
				\ar[lll,shift right=1.5ex,"\adjtocmd"{swap}]	&
\end{tikzcd}\]

The adjunction in the bottom-front is the 2-adjunction from Theorem~\ref{thm:cmd-adj}.
The one in the bottom-back is the biadjunction from Theorem~\ref{thm:cmd-adjloo}.

The left-hand biequivalence is proved in Theorem~\ref{eqv-wccmd-cmpct},
the right-hand one is proved in Theorem~\ref{eqv-wccmd-jdtt},
which also proves the top 2-adjunction.

\begin{remark}\label{eqv-is-fibr}
All top 2-categories in the diagram above have a forgetful to $\fibCat$ (for a \gcwf{} $(u,\du,\ldots)$ it is $u$)
and, as it is clear from their definitions in~\eqref{cmpct2wccmd},~\eqref{wccmd2cmpct},~\eqref{cor:wccmd2jdtt}, and~\eqref{cor:jdtt2wccmd},
all top 2-functors commute (strictly) with these forgetful 2-functors.

In particular, it follows that the whole diagram restricts to the 2-full sub-2-categories
on objects whose base category has a terminal object\ and on morphisms preserving it.
\end{remark}

\begin{remark}\label{dsc-gcwf}
In the 2-category $\dscjdttCat$ of \emph{discrete} \gencwf{ies},
\ie those whose fibrations $u$ and $\du$ are discrete fibrations,
every loose \gcwf{} morphism is a \gcwf{} morphism
since in discrete fibrations all vertical isos are identities.
Note however that it still makes sense to distinguish between lax, pseudo, and strict morphisms,
since the mate of $\id\colon \Sigma'\dot{H} = H\Sigma$ need not be vertical.
In fact, it is vertical if and only if the morphism strictly preserves comprehensions.
Let us identify categories with families with discrete \gencwf{ies}:
\[
\cwfCat := \dscjdttCat,
\hspace{2em}
\cwfpsCat := \dscjdttpsCat,
\hspace{2em}
\cwfstrCat := \dscjdttstrCat.
\]

Consider also the full sub-2-category $\dsccompcatCat$ of $\compcatCat$
on those objects with discrete fibration.
Note that 2-cells $(\psi,\phi)$ in $\dsccompcatCat$ and
2-cells $(\psi,\phi,\dot{\phi})$ in $\dscjdttCat$
are determined by $\psi$
since all components of $\phi$ and $\dot{\phi}$ have to be cartesian,
and cartesian lifts are unique in discrete fibrations.
In particular, the 2-category $\cwfpsCat$
(together with terminal objects in base categories and morphisms preserving terminal objects, see~\eqref{eqv-is-fibr})
is the one described by Uemura in~\citey[Example~5.21]{Uemura2023}:
the Beck-Chevalley condition~\cite[Definition~3.13]{Uemura2023}
requires the mate of $\id\colon \Sigma'\dot{H} = H\Sigma$ to be invertible.
More general morphisms between categories with families,
the pseudo cwf-morphisms of Clairambault and Dybjer in particular, are discussed in \eqref{ps-cwfmorph}.

In~\eqref{disceqv}, we show that
the biequivalence in~\eqref{eqv-compc-gcwf} restricts to an adjoint 2-equivalence between $\dsccompcatCat$ and $\cwfCat$,
which further restricts to their 2-full sub-2-categories on pseudo and strict morphisms and, in particular,
yields the classical equivalence by Hofmann between discrete comprehension categories and categories with families.
\end{remark}

\begin{remark}\label{compcat-jdtt:split}
The biequivalence also restricts if we require
\begin{enumerate}
\item
all the components on $\fibCat$ of the 0-cells to come equipped with a split cleavage and, for \gcwfs{}, that the functor $\Sigma$ preserves the cleavage, and
\item
all the components on $\fibCat$ of the 1-cells to preserve the cleavage.
\end{enumerate}
Indeed, the 2-functors $\cmpcttowccmd$, $\wccmdtocmpct$, and $\jdttpstowccmd$
fix the component on $\fibCat$ of the structures involved.
The 2-functor $\wccmdtojdttps$ fixes the first component on $\fibCat$
and its second $\fibCat$ component is the fibration of coalgebras.
As observed in~\ref{wccmd2jdtt:split},
the fibration of coalgebras of a \wccmd{} on a split fibration is also split,
and given a lax morphism of \wccmd[s] $(C,H,\theta)$ such that $(C,H)$ preserves the cleavage,
the pair $(C,\coal(H,\theta))$ also preserves the cleavage.
\end{remark}

\subsection{The biequivalence between comprehension categories and \wccmd[s]}
\label{wccmd_v_compcat}

First of all, we prove the 2-equivalence suggested in \cite[9.3.4]{jacobs1999categorical}. For the following result we need to assume that the underlying fibrations of comprehension categories and \wccmd[s] are cloven. Morphisms, however, are not required to preserve cleavages.

\begin{lemma}\label{cmpct2wccmd}
There is a 2-functor $\cmpcttowccmd \colon \compcatCat \to \wccmdCat$.

This 2-functor restricts to the wide sub-2-category on the pseudo morphisms.
\end{lemma}

\begin{proof}
	Let $(p,\chi)\colon \mathcal{E}\to \ctg{B}\due$ be a comprehension category
	together with a cleavage for $p$.
	For each $E$ in $\ctg{E}$,
	consider the chosen reindexing of $E$ along its comprehension $\chi_E$ as below.
	\begin{equation}\label{cmpct2wccmd-cmd}
	\begin{tikzcd}[ampersand replacement=\&]
		{K_\chi E} \& E \& {\mathcal{E}} \\
		{X_E} \& pE \& {\ctg{B}}
		\arrow["{\chi_E}", from=2-1, to=2-2]
		\arrow["{\overline{\chi_E}}", from=1-1, to=1-2]
		\arrow["p", from=1-3, to=2-3]
	\end{tikzcd}
	\end{equation}
	Since cartesian lifts are defined by a universal property,
	$K_\chi$ extends to an endofunctor $K_\chi$ on $\mathcal{E}$.
	Moreover, $K_\chi$ is copointed
	because the transformation $\epsilon_E := \overline{\chi_E}$ is natural
	by the very definition of $K_\chi$ on arrows.
	It satisfies~(\ref{wccmd-def}.\ref{wccmd-def:car}) by construction
	and~(\ref{wccmd-def}.\ref{wccmd-def:pb})
	by the fact that $\chi$ preserves cartesian arrows.
	Therefore $(K_\chi,\epsilon)$ is a \wccmd.
	
	A lax morphism of comprehension categories $(C,H,\zeta)\colon (p,\chi)\to(p',\chi')$
	induces a 1-cell $(C,H)\colon p \to p'$ in $\fibCat$ by its very definition.
	To obtain a lax morphism of \wccmd[s] $(C,H,\theta)\colon K_\chi \to K_{\chi'}$,
	it only remains to provide $\theta\colon HK_{\chi} \Rightarrow K_{\chi'}H$
	that makes $(H,\theta)$ into a lax morphisms of comonads.
	For $E$ over $X$, the component $\theta_E$ can be obtained,
	using the fact that $\epsilon'_{HE}$ is cartesian,
	as the universal arrow induced by $H\epsilon_E$ as in the diagram below.
	\begin{equation}\label{cmpct2wccmd-theta}
	\begin{tikzcd}
		{K'HE}
		&&&	BX_E	&
		\\
		HKE & HE
		&&	{X'_{HE}} & BX
		\arrow["{H(\epsilon_E)}", curve={height=-12pt}, from=1-1, to=2-2]
		\arrow["{B(\chi_E)}", curve={height=-12pt}, from=1-4, to=2-5]
		\arrow["{\epsilon'_{HE}}", from=2-1, to=2-2]
		\arrow["{\theta_E}"',dashed, from=1-1, to=2-1]
		\arrow["{\chi'_{HE}}", from=2-4, to=2-5]
		\arrow["\zeta_E"', from=1-4, to=2-4]
	\end{tikzcd}
	\end{equation}
	Naturality of $\theta$ follows from that of $\epsilon$ and $\epsilon'$
	using again the fact that the components of $\epsilon'$ are cartesian.
	Finally, $\theta$ commutes with the counits by definition,
	and it does so with the comultiplications since these are canonically
	determined by counits~(\ref{wccmd-rmk}.\ref{wccmd-rmk:def}).
	This action is clearly functorial in $H$ and $C$,
	and it is so in $\zeta$ since $\theta$ is defined by a universal property.
	
	It is clear from~\eqref{cmpct2wccmd-theta} that $\theta_E$ is invertible if (and only) $\zeta_E$ is invertible.

	To conclude the construction,
	we show that a 2-cell
	$(\psi,\phi) \colon (B_1,H_1,\zeta^1) \Rightarrow (B_2,H_2,\zeta^2)$ in $\compcatCat$ is also a 2-cell
	$(\psi,\phi)\colon(B_1,H_1,\theta^1) \Rightarrow (B_2,H_2,\theta^2)$ in $\wccmdCat$.
	As $(\psi,\phi)$ is, in particular, a 2-cell in $\fibCat$,
	it only remains to check that
	$K_{\chi'}\phi \circ \theta^1 = \theta^2 \circ \phi K_{\chi}$.
	This amounts to verifying that, for every $E$ over $X$,
	the left-hand square in the left-hand diagram below commutes.
	\[\begin{tikzcd}[row sep=1.2em]
		{H_1K_{\chi}E} &&&& {B_1X_E} \\
		& {K_{\chi'}H_1E} & {H_1E} &&& {X'_{H_1E}} & {B_1X} \\
		{H_2K_{\chi}E} &&&& {B_2X_E} \\
		& {K_{\chi'}H_2E} & {H_2E} &&& {X'_{H_2E}} & {B_2X} \\
		&& {} && {}
		\arrow["{\epsilon'_{H_1E}}"', from=2-2, to=2-3]
		\arrow["{\theta^1_E}"', from=1-1, to=2-2]
		\arrow["{H_1\epsilon_E}"{pos=0.7}, curve={height=-12pt}, from=1-1, to=2-3]
		\arrow["{\phi_E}"{pos=0.4}, from=2-3, to=4-3]
		\arrow["{\epsilon'_{H_2E}}"', from=4-2, to=4-3]
		\arrow["{\phi_{K_{\chi}E}}"', from=1-1, to=3-1]
		\arrow["{\theta^2_E}"', from=3-1, to=4-2]
		\arrow["{H_2\epsilon_E}"{description, pos=0.8}, curve={height=-12pt}, from=3-1, to=4-3]
		\arrow["{K_{\chi'}\phi_E}"{description, pos=0.39}, from=2-2, to=4-2, crossing over]
		\arrow["{\zeta^1_E}"', from=1-5, to=2-6]
		\arrow["{\chi'_{H_1E}}"', from=2-6, to=2-7]
		\arrow["{B_1\chi_E}"{pos=0.8}, curve={height=-12pt}, from=1-5, to=2-7]
		\arrow["{\psi_{X_E}}"', from=1-5, to=3-5]
		\arrow["{\zeta^2_E}"', from=3-5, to=4-6]
		\arrow["{\chi'_{H_2E}}"', from=4-6, to=4-7]
		\arrow["{B_2\chi_E}"{description, pos=0.8}, curve={height=-12pt}, from=3-5, to=4-7]
		\arrow[from=2-6, to=4-6, crossing over]
		\arrow["{\psi_X}", from=2-7, to=4-7]
		\arrow["{p'}", maps to, from=5-3, to=5-5]
	\end{tikzcd}\]
	But this follows from the fact that $\epsilon'_{H_2E}$ is cartesian
	once we show that the other faces and the right-hand diagram commute.
	The right-hand diagram commutes by \ref{cmpcat-2cell-sq},
	the two triangles commute by definition of $\theta$~\eqref{cmpct2wccmd-theta}, and
	the back and front squares by naturality of $\phi$ and $\epsilon'$, respectively.
	Functoriality is trivial.
\end{proof}

\begin{remark}\label{cmpct2wccmd-nostr}
Note that the 2-functor $\cmpcttowccmd$ does not necessarily map a strict morphism of comprehension categories to a strict morphisms of \wccmd[s].
Indeed, it is clear from the definition of $\theta$ in~\eqref{cmpct2wccmd-theta}
that, if $\zeta$ is an identity, $\theta$ is only forced to be a vertical iso.
\end{remark}

\begin{lemma}\label{wccmd2cmpct}
There is a 2-functor $\wccmdtocmpct \colon \wccmdCat \to \compcatCat$,
which restricts to the wide 2-full sub-2-categories on the pseudo and invertible morphisms.
\end{lemma}

\begin{proof}
On objects, it suffices to map a pair $(p\colon \mathcal{E}\to \ctg{B},K)$ to $\chi\colon \mathcal{E}\to\ctg{B}\due$, $\chi (E) := p\epsilon_E$.

To define its action on a 1-cell $(C,H,\theta)$, we use $\theta$ to induce a suitable
$\zeta \colon C\due\chi \Rightarrow \chi' H$ as follows:
\[\begin{tikzcd}
	CpKE && CpE \\
	{p'HKE} \\
	{p'K'HE} && {p'HE}
	\arrow["{Cp\epsilon_E}", from=1-1, to=1-3]
	\arrow["{p'\epsilon'_{HE}}", from=3-1, to=3-3]
	\arrow["{p'\theta_E}"', from=2-1, to=3-1]
	\arrow["{\id}"', from=1-1, to=2-1]
	\arrow["{\id}", from=1-3, to=3-3]
\end{tikzcd}\]
on the top row we read $C\due\chi_E = Cp\epsilon_E$, on the bottom $\chi'H = p'\epsilon'_{HE}$,
and the square commutes because, by hypothesis, $Cp=p'H$.
It follows, thanks also to~\eqref{cmpct-morph-triang},
that we can define $\zeta_E := p'\theta_E$.
With this definition, proving that
a 2-cell $(\psi,\phi)$ in $\wccmdCat$ is also a 2-cell in $\compcatCat$ is straightforward
using that the $\theta$'s commutes with the counits.
Functoriality is clear.
\end{proof}

Jacobs proves in~\citey[Theorem~9.3.4]{jacobs1999categorical} that \wccmd[s] are in bijection with comprehension categories.
We extend that result to lax morphisms.

\begin{theorem}\label{eqv-wccmd-cmpct}
The two 2-functors $\wccmdtocmpct \colon \wccmdCat \leftrightarrows \compcatCat \,\colon\! \cmpcttowccmd$
give rise to an adjoint biequivalence such that $\wccmdtocmpct \circ \cmpcttowccmd = \Id$.

The biequivalence restricts to the wide 2-full sub-2-categories on the pseudo morphisms.
\end{theorem}
\begin{proof}
We have to show that $\wccmdtocmpct \circ \cmpcttowccmd = \Id$ and that there is a natural iso
$\boldsymbol{\xi}\colon\cmpcttowccmd\wccmdtocmpct\natiso\Id$ such that
\[
\wccmdtocmpct \boldsymbol{\xi} = \idd_{\wccmdtocmpct} \tand \boldsymbol{\xi} \cmpcttowccmd = \idd_{\cmpcttowccmd}.
\]

The first equation follows from
$p(\overline{\chi_E}) = \chi_E$ and
$p(\theta_E) = \zeta_E$,
which hold by constructions~\eqref{cmpct2wccmd-cmd} and~\eqref{cmpct2wccmd-theta}, respectively,
and the fact that both 2-functors fix the 2-cells.

To obtain the natural iso $\boldsymbol{\xi}$ recall that,
in $\cmpcttowccmd \wccmdtocmpct(p,K,\epsilon)$,
the counit at $E$ is defined as the (chosen) cartesian lift of $p\epsilon_E$~\eqref{cmpct2wccmd-cmd}.
As $\epsilon_E$ is also cartesian over $p\epsilon_E$ and into $E$,
it follows that there is a unique vertical invertible arrow $\xi'_E$ between them.
The component $\boldsymbol{\xi}_{(p,K)}$ is then the (invertible) morphism of comonads 
$(\Id,\Id,\xi')$.
Naturality is ensured by the uniqueness of these vertical isos $\xi'$.
The first equation is then clear.
The second one holds since,
if $(p,K)$ is in the image of $\cmpcttowccmd$,
also $\epsilon_E$ is a chosen cartesian lift of $p$.
As $\wccmdtocmpct$ and $\cmpcttowccmd$ fix $p$,
it is the same (chosen) cartesian lift
as the one in $\cmpcttowccmd \wccmdtocmpct(p,K)$.
\end{proof}

\subsection{The biequivalence between \wccmd[s] and \gencwf{ies}}
\label{wccmd_v_jdtt}

Here we use the two adjunctions from Section~\ref{sec:cmd-adj}.

\begin{theorem}\label{thm:wccmd-jdtt}\label{eqv-wccmd-jdtt}\label{thm:cmdeqjdtts}
The biadjunction in~\eqref{thm:cmd-adjloo} lifts to an adjoint biequivalence
$\jdttpstowccmd \equiv \wccmdtojdttps$ on the left-hand below 
whose counit components are identities.
In particular, $\jdttpstowccmd \circ \wccmdtojdttps =\Id$.

The biequivalence restricts along $\jdttCat \hookrightarrow \jdttlCat$
to a 2-reflection $\jdtttowccmd \dashv \wccmdtojdtt$ on the right-hand below,
lifting the 2-adjunction in~\eqref{thm:cmd-adj}.
\[\begin{tikzcd}[column sep=3em]
\wccmdCat	\ar[r,bend right=1.3em,"\wccmdtojdttps"'] \ar[r,phantom,"\equiv"{description}]
&	\jdttlCat	\ar[l,bend right=1.3em,"\jdttpstowccmd"']
\end{tikzcd}
\hspace{5em}
\begin{tikzcd}[column sep=3em]
\wccmdCat	\ar[r,bend right=1.3em,"\wccmdtojdtt"'] \ar[r,phantom,"\adj{-90}"{description}]
&	\jdttCat	\ar[l,bend right=1.3em,"\jdtttowccmd"']
\end{tikzcd}\]

The biequivalence also restricts to the wide 2-full sub-2-categories on the pseudo morphisms,
and the 2-reflection restricts to the wide 2-full sub-2-categories on the pseudo and on the strict morphisms.
\[\begin{tikzcd}[column sep=3em]
\wccmdpsCat	\ar[r,bend right=1.3em,"\wccmdtojdttps"'] \ar[r,phantom,"\equiv"{description}]
&	\jdttlpsCat	\ar[l,bend right=1.3em,"\jdttpstowccmd"']
\end{tikzcd}\]
\vspace{-1em}
\[\begin{tikzcd}[column sep=3em]
\wccmdpsCat	\ar[r,bend right=1.3em,"\wccmdtojdtt"'] \ar[r,phantom,"\adj{-90}"{description}]
&	\jdttpsCat	\ar[l,bend right=1.3em,"\jdtttowccmd"']
\end{tikzcd}
\hspace{8em}
\begin{tikzcd}[column sep=3em]
\wccmdstrCat	\ar[r,bend right=1.3em,"\wccmdtojdtt"'] \ar[r,phantom,"\adj{-90}"{description}]
&	\jdttstrCat	\ar[l,bend right=1.3em,"\jdtttowccmd"']
\end{tikzcd}\]

\vspace{1em}
\end{theorem}

The rest of the section is devoted to the proof of Theorem~\ref{thm:wccmd-jdtt}.
We begin with two lemmas ensuring that the 2-functors
$\cmdtoadj$ and $\adjtocmd$ lift to 2-functors
between $\jdttCat$ and $\wccmdCat$.
We begin with $\cmdtoadj$.

\begin{lemma}\label{lem:wccmd2jdtt}
\hfill
\begin{enumerate}
\item\label{lem:wccmd2jdtt:obj}
If $(p,K,\epsilon,\nu)$ is a \wccmd,
then $p\emfrg{K} \colon \coal(K) \to \mathcal{B}$ is a fibration.
\item\label{lem:wccmd2jdtt:mor}
If $(B,H,\theta) \colon (p,K,\epsilon,\nu) \to (p',K',\epsilon',\nu')$
is a lax morphism of \wccmd[s],
then $(B,\coal(H,\theta)) \colon p\emfrg{K} \to p'\emfrg{K'}$
is a morphism of fibrations.
\end{enumerate}
\end{lemma}

\begin{proof}
\eqref{lem:wccmd2jdtt:obj}.
Consider the Eilenberg--Moore adjunction associated to $(K,\epsilon,\nu)$
\[\begin{tikzcd}
{\coal(K)} && {\ctg{E}}
\arrow[""{name=0, anchor=center, inner sep=0}, "\emrgt{K}"', shift right=2, from=1-3, to=1-1]
\arrow[""{name=1, anchor=center, inner sep=0}, "\emfrg{K}"', shift right=2, from=1-1, to=1-3]
\arrow["\dashv"{anchor=center, rotate=90}, draw=none, from=1, to=0]
\end{tikzcd}\]
and let $e\colon E \car KE$ a coalgebra, $\sigma \colon X \to pE$
and $s\colon E\sigma \car E$ a $p$-cartesian  lift of $\sigma$.
To have a cartesian lift of $\sigma$ at $e$ in $\coal(K)$,
it is enough to find an arrow $e\sigma$ which is a coalgebra
and such that the left-hand square in~\eqref{lem:wccmd2jdtt:obj:lift} commutes.
\begin{equation}\label{lem:wccmd2jdtt:obj:lift}
\begin{tikzcd}[ampersand replacement=\&]
E\sigma \& E \&\& {K(E\sigma)} \& {K(E)} \\
{K(E\sigma)} \& {K(E)} \&\& E\sigma \& E
\arrow["s"', from=1-1, to=1-2]
\arrow["{K(s)}"', from=2-1, to=2-2]
\arrow["e", from=1-2, to=2-2]
\arrow["{e\sigma}"',dashed, from=1-1, to=2-1]
\arrow["{K(s)}"', from=1-4, to=1-5]
\arrow["{\epsilon_E}", from=1-5, to=2-5]
\arrow["{\epsilon_{E\sigma}}"', from=1-4, to=2-4]
\arrow["s"', from=2-4, to=2-5]
\end{tikzcd}
\end{equation}

The right-hand square in~\eqref{lem:wccmd2jdtt:obj:lift}
is a pullback by~(\ref{wccmd-rmk}.\ref{wccmd-rmk:pb'}),
therefore the span
\[\begin{tikzcd}
E\sigma & E\sigma & E
\arrow["{ e\circ s}", from=1-2, to=1-3]
\arrow["id"', from=1-2, to=1-1]
\end{tikzcd}\]
induces a (unique) section $e\sigma$ of $\epsilon_{E\sigma}$
which makes the left-hand square in~\eqref{lem:wccmd2jdtt:obj:lift} commute.
It is a coalgebra by~(\ref{wccmd-rmk}.\ref{wccmd-rmk:def}).

\eqref{lem:wccmd2jdtt:mor}.
We have
$p'\emfrg{K'}\coal(H,\theta) = p'H\emfrg{K} = B p\emfrg{K}$,
and $\coal(H,\theta)$ preserves cartesian arrows by naturality of $\theta$.
\end{proof}

\begin{remark}\label{wccmd2jdtt:split}
The proof of (\ref{lem:wccmd2jdtt}.\ref{lem:wccmd2jdtt:obj}) above shows in particular that a cleavage for $p$ induces a cleavage for $p\emfrg{K}$.
It is clear that $\emfrg{K}$ maps one to the other.
It is also easy to see, using functoriality of $K$, that a split cleavage induces a split cleavage.

If $(C,H,\theta)$ is a morphism of \wccmd[s] such that $(C,H)$ preserves the cleavage, then so does $(C,\coal(H,\theta))$.
For this, it is enough to show that $\theta_{E\sigma} \circ He\sigma$ equals the chosen reindexing of $\theta_E \circ He$ over $C\sigma$.
Since $(C,H)$ preserves the cleavage,
the latter is the (unique) dashed arrow making the square below commute
\[\begin{tikzcd}
H(E\sigma)	\ar[d,dashed] \ar[r,"{H(e\sigma)}"]
&	HE	\ar[d,"{\theta_E \circ He}"]
\\
K'H(E\sigma)	\ar[r,"{K'H(e\sigma)}"]	&	K'HE
\end{tikzcd}\]
The claim follows from the fact that $\theta_{E\sigma} \circ H(e\sigma)$ also makes that square commute.
\end{remark}

\begin{corollary}\label{cor:wccmd2jdtt}
The 2-functor $\cmdtoadjloo \colon \cmdCat \to \adjLlCat$ lifts
to a 2-functor
\[\begin{tikzcd}[column sep=10em, row sep=1ex]
\wccmdCat	\ar[r,"\wccmdtojdttps"]	&	\jdttlCat
\\
(p,K,\epsilon,\nu)	\ar[d,"{(C,H,\theta)}"{swap,name=ddom}, bend right=3em, shift right=2, "\qquad"] 	\ar[d,"{(C',H',\theta')}"{name=dom,right}, bend left=3em, shift left=2,"\qquad"] \arrow[Rightarrow,from=ddom,to=dom,shorten=1mm, "{(\gamma,\phi)}"]
&	(p,p\emfrg{K},\cmdtoadjloo(K,\epsilon,\nu))	\ar[d,"{(C,H,\coal(H,\theta))}"{swap,name=cod,left},"\qquad",bend right=3em,shift right=6.5]
\ar[d,"{(C',H',\coal(H',\theta))}"{name=ccod},"\qquad",bend left=3em, shift left=6.5] \arrow[Rightarrow,from=cod,to=ccod,shorten=1mm, "{(\gamma,\phi,\coal(\phi))}", crossing over]
\\[2em]
(p',K',\epsilon',\nu')	&	(p',p'\emfrg{K'},\cmdtoadjloo(K',\epsilon',\nu'))
\ar[from=dom, to=cod,mapsto]
\end{tikzcd}\]
This 2-functor restricts to 2-functors between the wide 2-full sub-2-categories on the pseudo and strict morphisms.
\end{corollary}

\begin{proof}
First, we need to verify that $(p,p\emfrg{K},\cmdtoadjloo(K,\epsilon,\nu))$
is a \gcwf.
We already know from~\eqref{sec:cmd-adj:la} that $\cmdtoadjloo(K,\epsilon,\nu)$ is an adjunction.
The functor $p\emfrg{K}$ is a fibration by (\ref{lem:wccmd2jdtt}.\ref{lem:wccmd2jdtt:obj}).
It only remains to show that the components of the unit and counit of $\emfrg{K}\dashv \emrgt{K}$ are cartesian arrows.
For the counit this holds by assumption,
and the component of the unit at a coalgebra is the coalgebra structure map,
which is cartesian by \ref{wccmd-coalg}.

Given a lax morphism of \wccmd[s]
$(C,H,\theta) \colon (p,K,\epsilon,\nu) \to (p',K',\epsilon',\nu')$,
we have that $(C,H)$ is a morphism of fibrations by assumption,
$(C,\coal(H,\theta))$ is a morphism of fibrations by (\ref{lem:wccmd2jdtt}.\ref{lem:wccmd2jdtt:mor}),
and $(H,\coal(H,\theta)) = \cmdtoadj(H,\theta)$ is a left morphism of adjunctions
by~\eqref{sec:cmd-adj:ra}.
This proves that $(C,H,\coal(H,\theta))$ is a \gcwf{} morphism.

Given a 2-cell
$(\gamma,\phi)\colon (C_1,H_1,\theta_1)\Rightarrow(C_2,H_2,\theta_2)$
in $\wccmdCat$,
the pairs $(\gamma,\phi)$ and $(\gamma,\coal(\phi))$ are clearly 2-cells in $\fibCat$,
and $(\phi,\coal(\phi))$ is 2-cell in $\adjLlCat$ by~\eqref{sec:cmd-adj:ra}.
It follows that $(\gamma,\phi,\coal(\phi))$ is a 2-cell in $\jdttlCat$.

The 2-functor restricts since $\cmdtoadj$ does by~\eqref{sec:cmd-adj:la}.
\end{proof}

We now turn to the 2-functor $\adjtocmd$ from adjunctions to comonads.

\begin{lemma}\label{lem:jdtt2wccmd}
If $(u,\du,\Sigma\dashv\Delta)$ is a \gcwf,
then for every cartesian arrow $f \colon A \car B$ in $\uu$
the square
\[\begin{tikzcd}
X.A	\ar[d,"\du\Delta f"{swap}] \ar[r,"{u\epsilon_A}"]
&	X	\ar[d,"uf"]
\\
Y.B	\ar[r,"{u\epsilon_B}"]	&	Y
\end{tikzcd}\]
is a pullback in $\ctg{B}$.
\end{lemma}

\begin{proof}
Let $k \colon Z \to X$ and $h \colon Z \to pKB$ be such that
$(u\epsilon_B) h = (uf) k$ and
consider a cartesian arrow $b \colon M \car \Delta B$ in $\duu$ over $h$.
The arrow induced by $h$ and $k$ will be the image under $\du$ of a (cartesian) arrow
$d \colon M \car \Delta A$ in $\duu$ such that $(\Delta f)d = b$.
Note first that, since $f$ is cartesian,
there is a unique arrow $a \colon \Sigma M \to A$ in $\uu$ over $k$
such that the left-hand diagram in~\eqref{lem:jdtt2wccmd:transpsq} commutes.
In particular, $a$ is cartesian
since $f$ and $\epsilon_B (\Sigma b) \colon \Sigma M\car B$ are.
\begin{equation}\label{lem:jdtt2wccmd:transpsq}
\begin{tikzcd}
	{\Sigma M} & A \\
	{\Sigma\Delta B} & B
	\arrow["a", from=1-1, to=1-2]
	\arrow["f", from=1-2, to=2-2]
	\arrow["{\epsilon_{B}}"', from=2-1, to=2-2]
	\arrow["{\Sigma b}"', from=1-1, to=2-1]
\end{tikzcd}
\hspace{4em}
\begin{tikzcd}
	M & {\Delta A} \\
	{\Delta B} & {\Delta B}
	\arrow["{a^\#}", from=1-1, to=1-2]
	\arrow["{\Delta f}", from=1-2, to=2-2]
	\arrow["b"', from=1-1, to=2-1]
	\arrow["{\id_{\Delta B}}"', from=2-1, to=2-2]
\end{tikzcd}
\hspace{4em}
\begin{tikzcd}
	{\Sigma M} & A \\
	{\Sigma\Delta A} & A
	\arrow["a", from=1-1, to=1-2]
	\arrow["{\id_A}", from=1-2, to=2-2]
	\arrow["{\epsilon_{A}}"', from=2-1, to=2-2]
	\arrow["{\Sigma a^\#}"', from=1-1, to=2-1]
\end{tikzcd}
\end{equation}
Transposing the left-hand square in~\eqref{lem:jdtt2wccmd:transpsq} yields the central one,
while transposing a trivial square involving $a^\#$ yields the right-hand one.
It follows that all three squares in~\eqref{lem:jdtt2wccmd:transpsq} commute.

Define
\[
d := a^\# \colon M \to \Delta A,
\]
which is cartesian because $d=\Delta a\circ \eta_M$, the unit is cartesian and $\Delta$ preserves cartesian maps by~\eqref{deltacartesian}.
Commutativity of the central and right-hand square
in~\eqref{lem:jdtt2wccmd:transpsq} entails that
$(\du\Delta f) (\du d) = h$ and $(u\epsilon_{A}) (\du d) = k$,
respectively.
We are left to prove that $\du d$ is the unique such.

Let $l \colon Z \to X.A$ be such that
$(\du\Delta f)l = h$ and $(u\epsilon_{A})l = k$.
Since $\Delta f$ is cartesian, there is a unique arrow
$l' \colon M \to \Delta A$ over $l$ such that $(\Delta f)l' = b$.
Transposing as above yields $f l'^\# = \epsilon_B (\Sigma b) = f a$.
As $u (l'^\#) = u (\epsilon_A (\Sigma l'))= k$,
it follows that $l'^\# = a$, and thus $l' = d$.
\end{proof}

\begin{corollary}\label{cor:jdtt2wccmd}
The 2-functor $\adjlootocmd \colon \adjLlCat \to \cmdCat$ lifts
to a 2-functor
\[\begin{tikzcd}[column sep=10em, row sep=1ex]
\jdttlCat	\ar[r,"\jdttpstowccmd"]	&	\wccmdCat
\\
(u,\du,\Sigma\dashv\Delta)	\ar[d,"{(C,F,G,\zeta)}"{swap,name=ddom}, bend right=3em,shift right=1, "\qquad"] 	\ar[d,"{(C',F',G',\zeta')}"{name=dom,right}, bend left=3em,shift left=1,"\qquad"] \arrow[Rightarrow,from=ddom,to=dom,shorten=1mm, "{(\gamma,\phi,\psi)}"]
&	(u,\adjlootocmd(\Sigma,\Delta))	\ar[d,"{(C,\adjlootocmd(F,G,\zeta))}"{swap,name=cod,left},"\qquad",bend right=3em]
\ar[d,"{(C',\adjlootocmd(F',G',\zeta'))}"{name=ccod},"\qquad",bend left=3em] \arrow[Rightarrow,from=cod,to=ccod,shorten=1mm, "{(\gamma,\phi)}"]
\\[2em]
(u',\du',\Sigma'\dashv\Delta')&	(u',\adjlootocmd(\Sigma',\Delta'))
\ar[from=dom, to=cod,mapsto]
\end{tikzcd}\]
The 2-functor $\jdttpstowccmd$ restricts to a 2-functor
between the wide 2-full sub-2-category on the pseudo morphisms.
Its restriction $\jdtttowccmd$ to $\jdttCat$ also restricts to a 2-functor
between the wide 2-full sub-2-category on the strict morphisms.
\end{corollary}
\begin{proof}
We need to verify that, when $(u,\du,\Sigma\dashv\Delta)$ is a \gcwf,
the comonad $\adjtocmd(\Sigma,\Delta)$ is a \wccmd{} with fibration $u$.
Condition~\ref{wccmd-def:car} in~\eqref{wccmd-def} is satisfied
since the counit is cartesian by assumption.
Condition~\ref{wccmd-def:pb} in~\eqref{wccmd-def} follows from~(\ref{wccmd-rmk}.\ref{wccmd-rmk:pb'}) and~\eqref{lem:jdtt2wccmd}.

To verify that $(C,\adjlootocmd(F,G,\zeta))$ is a morphism of \wccmd[s]
whenever $(C,F,G,\zeta)$ is a 1-cell in $\jdttlCat$
note that the functor component of $\adjlootocmd(F,G,\zeta)$ is $F$
by construction~\eqref{ext-loose}.
But $(C,F)$ is a morphism of fibrations by assumption, and
we already know that $\adjlootocmd(F,G,\zeta)$ is a lax morphism of comonads.

Given a 2-cell $(\gamma,\phi,\psi)\colon (C_1,F_1,G_1,\zeta_1) \to (C_2,F_2,G_2,\zeta_2)$ in $\jdttlCat$,
it is clear that
$(\gamma,\phi) \colon (C_1,\jdttpstowccmd(F_1,G_1,\zeta_1)) \to (C_2,\jdttpstowccmd(F_2,G_2,\zeta_2))$
is a 2-cell in $\wccmdCat$,
since $\adjlootocmd(\phi,\psi) = \phi$ is a 2-cell in $\cmdCat$ by~\eqref{ext-loose}.

The 2-functor restricts as stated because $\adjlootocmd$ does, see~\eqref{sec:cmd-adj:la} and~\eqref{ext-loose}.
\end{proof}

\begin{proof}[of Theorem \ref{thm:wccmd-jdtt}]
From \ref{cor:wccmd2jdtt} and \ref{cor:jdtt2wccmd} we have two 2-functors
\[
\wccmdtojdttps \colon \wccmdCat \leftrightarrows \jdttlCat \,\colon\! \jdttpstowccmd
\]
We begin by showing that they form a biadjunction by lifting the biadjunction from~\eqref{thm:cmd-adjloo}.
The 2-adjunction involving $\jdttCat$ will follow by restriction
along the inclusion $\jdttCat \hookrightarrow \jdttlCat$.

As in \ref{thm:cmd-adjloo},
the composite $\jdttpstowccmd \circ \wccmdtojdttps$ is the identity on $\wccmdCat$.
Next, we show that the unit $\unitadjcmd$ of $\adjtocmd\dashv\cmdtoadj$ from~\eqref{thm:cmd-adjloo} lifts to a pseudo-natural transformation
$\unitadjloocmd \colon \Id_{\wccmdCat} \to \jdttpstowccmd\circ\wccmdtojdttps$.
The component of $\unitadjcmd$ at an adjunction $\Sigma \dashv \Delta$
is the strict left morphism of adjunctions $(\Id,\emcmp{\Sigma,\Delta})$,
where $\emcmp{\Sigma,\Delta}$ is the canonical comparison functor described in~\eqref{sec:cmd-adj:la}.
Given a \gcwf{} $(u,\du,\Sigma\dashv\Delta)$,
the component of $\unitadjcmd$ at $\Sigma \dashv \Delta$
is the strict left morphism of adjunctions $(\Id,\emcmp{\Sigma,\Delta})$.
The functor $\emcmp{\Sigma,\Delta}$ preserves cartesian arrows
since $\Sigma$ does and coalgebra structure maps are cartesian by~\eqref{wccmd-coalg}.
It follows that
\[
\unitadjloocmd_{u,\du,\Sigma\dashv\Delta}
=(\Id_\mathcal{C},\unitadjcmd_{\Sigma\dashv\Delta})
=(\Id_\mathcal{C},\Id_\uu,\emcmp{\Sigma,\Delta})
\]
is a strict \gcwf{} morphism.
To see that
this choice is pseudo-natural in $(u,\du,\Sigma\dashv\Delta)$,
let $(C,F,G,\zeta) \colon (u,\du,\Sigma\dashv\Delta) \to (u',\du',\Sigma'\dashv\Delta')$
be a loose \gcwf{} morphism.
The required invertible 2-cell
$(C,\unitadjcmd_{u',\du'}(F,G,\zeta)) \to (C,\cmdtoadj\adjtocmd(F,G,\zeta))\unitadjcmd_{u,\du})$
is $(\id_C,\id_F,\hat{\zeta})$,
where $\hat{\zeta}$ is the natural iso from \ref{rem:liftunit}
(and $(\id_F,\hat{\zeta})$ is the pseudo naturality of $\unitadjcmd$~\eqref{thm:cmd-adjloo}).
Clearly, if $\zeta=\id$ so is $\hat{\zeta}$, and thus the invertible 2-cell,
meaning that $\unitadjloocmd$ is natural with respect to \gcwf{} morphisms.
This last fact proves that the biadjunction will restrict to a 2-adjunction between $\wccmdCat$ and $\jdttCat$.

The triangular identities for $\jdttpstowccmd \dashv \wccmdtojdttps$
follow immediately from those of $\adjtocmd \dashv \cmdtoadj$ in~\eqref{eq:cmd-adj:triangid}:
\begin{gather*}
\jdttpstowccmd\unitadjloocmd_{u,\du,\Sigma,\Delta}
= (\Id_\mathcal{C},\adjtocmd\unitadjcmd_{\Sigma,\Delta})
= \idd_{\jdttpstowccmd(u,\du,\Sigma\dashv\Delta)}
\\
\unitadjloocmd_{\wccmdtojdttps(K,\epsilon,\nu)}
= (\Id_\mathcal{B},\unitadjcmd_{\cmdtoadj(K,\epsilon,\nu)})
= \idd_{\wccmdtojdttps(K,\epsilon,\nu)}
\end{gather*}

It remains to show that the biadjunction $\jdttpstowccmd \dashv \wccmdtojdttps$
is in fact a biequivalence.
This amounts to showing that
each component of $\unitadjloocmd$ is an equivalence in $\jdttlCat$.
As $\unitadjloocmd$ is pseudo-natural,
so will be the family of its weak inverses.
To construct a weak inverse,
consider the functor $\eminvcmp{\Sigma,\Delta}\colon \coal(\Sigma\Delta)\to\duu$
and natural isos
$\zeta \colon \eminvcmp{\Sigma,\Delta}\emcmp{\Sigma,\Delta} \natiso \Id_{\duu}$
and $\xi \colon \emcmp{\Sigma,\Delta}\eminvcmp{\Sigma,\Delta} \natiso \Id_{\coal(\Sigma\Delta)}$
from~\eqref{lem:inverse_to_comparison} below.
The quadruple $(\Id,\Id_{\uu},\eminvcmp{\Sigma,\Delta},\emfrg{\Sigma\Delta}\xi)$
is a loose \gcwf{} morphism
$(u,u\emfrg{\Sigma\Delta},\emfrg{\Sigma\Delta},\emrgt{\Sigma\Delta}) \to (u,\du,\Sigma\dashv\Delta)$
since $\emfrg{\Sigma\Delta}\xi\colon \Sigma \eminvcmp{\Sigma,\Delta} \natiso \emfrg{\Sigma\Delta}$.
Note that the triple $(\id,\id,\xi)$ is an invertible 2-cell in $\jdttlCat$
from $(\Id,\Id,\emcmp{}\eminvcmp{},\emfrg{}\xi)$
to $(\Id,\Id,\Id_{\coal(\Sigma\Delta)})$.
Note also that the triple $(\id,\id,\zeta)$ is an invertible 2-cell in $\jdttlCat$
from $(\Id,\Id,\eminvcmp{}\emcmp{},\emfrg{}\xi\emcmp{})$
to $(\Id,\Id,\Id_{\duu})$ since
$\Sigma\zeta = \emfrg{}\xi\emcmp{}$.
This concludes the proof of the biequivalence.

To see that the biadjunction and the 2-reflection restrict as stated,
recall first that the 2-functors $\wccmdtojdttps$ and $\jdttpstowccmd$ restrict in all three cases, see~\eqref{cor:wccmd2jdtt} and~\eqref{cor:jdtt2wccmd}.
The unit $\unitadjloocmd$ restricts as well since its components are strict \gcwf{} morphisms.

The biequivalence restricts to pseudo morphisms
because each component of the mate of $\emfrg{\Sigma\Delta}\xi$
is invertible by~\eqref{lem:inverse_to_comparison}.
\end{proof}

In the next lemma we construct the weak inverse used in the proof of~\eqref{thm:wccmd-jdtt} above.
To prove the lemma we assume that the term fibration $\du$ is cloven
(since $\Sigma$ preserves cartesian maps, $u$ becomes cloven too).

\begin{lemma}\label{lem:inverse_to_comparison}
Let $(u,\du,\Sigma\dashv\Delta)$ be a \gencwf{y}.
There are a functor $\eminvcmp{\Sigma,\Delta}\colon \coal(\Sigma\Delta)\to\duu$
and natural isos
$\zeta \colon \eminvcmp{\Sigma,\Delta}\emcmp{\Sigma,\Delta} \natiso \Id_{\duu}$
and $\xi \colon \emcmp{\Sigma,\Delta}\eminvcmp{\Sigma,\Delta} \natiso \Id_{\coal(\Sigma\Delta)}$
making $\emcmp{\Sigma,\Delta}$ and $\eminvcmp{\Sigma,\Delta}$
into an adjoint equivalence of categories, meaning that
\[
\emcmp{\Sigma,\Delta}\zeta = \xi\emcmp{\Sigma,\Delta}\hspace{3em}\text{and}\hspace{3em}
\eminvcmp{\Sigma,\Delta}\xi = \zeta\eminvcmp{\Sigma,\Delta}.
\]
Moreover, each component of $\zeta$, $\xi$, and the mate of $\emfrg{\Sigma\Delta}\xi$ is vertical,
and the latter is also invertible.
\end{lemma}

\begin{proof}
The functor $\eminvcmp{\Sigma,\Delta}$ is defined on a coalgebra $h \colon A \to \Sigma\Delta A$ as the reindexing of $\Delta A$ along $uh$.
The action on a morphism of coalgebras $f$ is induced accordingly
using the cartesian lift $\overline{uh}$ that defines $\eminvcmp{\Sigma,\Delta}h$ as depicted below,
where both top squares, in $\duu$ and $\uu$ respectively, sit on the bottom square.
\begin{gather*}
\begin{tikzcd}[ampersand replacement=\&]
\eminvcmp{\Sigma,\Delta}k	\ar[d,dashed,"\eminvcmp{\Sigma,\Delta}f"{swap}] \ar[r,"\overline{uk}"]	\&	\Delta B	\ar[d,"\Delta f"]
\\
\eminvcmp{\Sigma,\Delta}h	\ar[r,"\overline{uh}"]	\&	\Delta A
\end{tikzcd}
\hspace{4em}
\begin{tikzcd}[ampersand replacement=\&]
B	\ar[d,"f"{swap}] \ar[r,"k"]	\&	\Sigma\Delta B	\ar[d,"\Sigma\Delta f"]
\\
A	\ar[r,"h"]	\&	\Sigma\Delta A
\end{tikzcd}
\\[2ex]
\begin{tikzcd}[ampersand replacement=\&]
Y	\ar[d,"uf"{swap}] \ar[r,"uk"]	\&	Y.B	\ar[d,"\du\Delta f"]
\\
X	\ar[r,"uh"]	\&	X.A
\end{tikzcd}
\end{gather*}

As the action on arrows is defined by a universal property,
functoriality of $\eminvcmp{\Sigma,\Delta}$ is straightforward.
It is also clear that $\eminvcmp{\Sigma,\Delta}$ preserves cartesian arrows.

Recall that $\emcmp{\Sigma,\Delta}a = \Sigma\eta_a$.
Therefore $\eminvcmp{\Sigma,\Delta}\circ\emcmp{\Sigma,\Delta}a$
is defined as the domain of a cartesian lift of $\du\eta_a$ into $\Delta\Sigma a$.
But the component $\eta_a$ of the unit at an object $a$ of $\duu$ is also cartesian into $\Delta\Sigma a$.
Therefore there is a unique vertical iso
$\zeta_a \colon
\eminvcmp{\Sigma,\Delta}\Sigma\eta_a \to a$
such that $\eta_a \zeta_a = \overline{\du \eta_a}$.
Naturality can be shown using that $\eta_a$ is cartesian.
It follows that
$\zeta \colon \eminvcmp{\Sigma,\Delta}\emcmp{\Sigma,\Delta} \natiso \Id_{\duu}$.

On the other hand, a coalgebra $h$ is cartesian over $uh$ and so is $\Sigma\overline{uh}$,
since $\Sigma$ preserves cartesian arrows.
It follows that there is a unique vertical iso
$\xi_h \colon \Sigma \eminvcmp{\Sigma,\Delta}h \to A$
such that $h \xi_h = \Sigma \overline{uh}$.
Again, since $h$ is cartesian, $\xi_h$ is natural in $h$,
and it follows that
$\xi \colon \Sigma \eminvcmp{\Sigma,\Delta} \natiso \emfrg{\Sigma\Delta}$.
To obtain a natural iso
$\emcmp{\Sigma,\Delta}\eminvcmp{\Sigma,\Delta} \natiso \Id_{\coal(\Sigma\Delta)}$,
it is enough to show that $\xi_h$ is in fact a morphism,
and thus an iso, of coalgebras from
$\Sigma\eta_{\eminvcmp{\Sigma,\Delta}h}$ to $h$.
This amounts to the commutativity of the square below.
\[\begin{tikzcd}
\Sigma \eminvcmp{\Sigma,\Delta}h
	\ar[d,"{\Sigma\eta_{\eminvcmp{}h}}"{swap}]
	\ar[r,"{\xi_h}"]
	\ar[dr,"{\Sigma\overline{uh}}"{description}]
&	A	\ar[d,"h"]
\\
\Sigma\Delta\Sigma \eminvcmp{\Sigma,\Delta}h	\ar[r,"{\Sigma\Delta\xi_h}"{swap}]
&	\Sigma\Delta A
\end{tikzcd}\]

The upper-right triangle commutes by definition of $\xi_h$.
The lower-left triangle is the image under $\Sigma$ of the left-hand square below,
which is the transpose under $\Sigma\dashv\Delta$ of the right-hand square.
\[\begin{tikzcd}
\eminvcmp{\Sigma,\Delta}h	\ar[d,"{\overline{uh}}"{swap}] \ar[r,"{\eta_{\eminvcmp{}h}}"]
&	\Delta\Sigma \eminvcmp{\Sigma,\Delta}h	\ar[d,"{\Delta\xi_h}"]
\\
\Delta A	\ar[r,"{\id_A}"{swap}]	&	\Delta A
\end{tikzcd}
\hspace{3em}
\begin{tikzcd}
\Sigma \eminvcmp{\Sigma,\Delta}h	\ar[d,"\Sigma\overline{uh}"{swap}] \ar[r,"{\id_{\eminvcmp{}h}}"]
&	\Sigma \eminvcmp{\Sigma,\Delta}h	\ar[d,"{\xi_h}"]
\\
\Sigma\Delta A	\ar[r,"{\epsilon_A}"{swap}]	&	A
\end{tikzcd}\]
The right-hand square commutes since
$\Sigma\overline{uh} = h\xi_h$ and $\epsilon_A h = \id_A$.
It follows that the other two squares commute as well.

To see that
$\emcmp{\Sigma,\Delta}\zeta = \xi \emcmp{\Sigma,\Delta}$  note that, for every $a \in \duu$,
$\Sigma\eta_a \circ \Sigma\zeta_a = \Sigma\overline{\du\eta_a}$
by definition of $\zeta_a$.
It follows that $\Sigma\zeta_a = \xi_{\Sigma\eta_a}$
as required.
The other equation
$\eminvcmp{\Sigma,\Delta}\xi = \zeta \eminvcmp{\Sigma,\Delta}$
follows from the fact that $\overline{uh}$ is cartesian and
\[
\overline{uh}\circ \eminvcmp{\Sigma,\Delta}\xi_h
\ =\ \Delta \xi_h \circ \overline{\du\eta_{\eminvcmp{}h}}
\ =\ \Delta \xi_h \circ \eta_{\eminvcmp{}h} \circ \zeta_{\eminvcmp{}h}
\ =\ \overline{uh}\circ\zeta_{\eminvcmp{}h},
\]
using definitions of $\eminvcmp{}$ and of $\zeta$,
and commutativity of the left-hand square above.

Finally, to see that the mate
$(\emfrg{}\xi)^\# \colon \eminvcmp{\Sigma,\Delta}\emrgt{\Sigma\Delta} \Rightarrow \Delta$
is a vertical natural iso,
note first that $(\emfrg{}\xi_{\emrgt{}A})\eta_{\eminvcmp{}\emrgt{}A} = \overline{u\nu_A}$,
which can be seen post-composing with $\Delta(\epsilon_{\Sigma\Delta A}\nu_A)$
and using the definition of $\xi$.
It follows that $(\emfrg{}\xi)^\#_A = (\Delta\epsilon_A)\overline{u\nu_A}$
is cartesian over the identity on $\du\Delta A$,
thus vertical and invertible.
\end{proof}

\subsection{Discrete and full comprehension categories}
\label{sec:disc-full-compcats}

Recall that we have defined $\cwfCat = \dscjdttCat$~\eqref{dsc-gcwf}.
As we show below,
the biequivalence~\eqref{eqv-compc-gcwf} restricts to categories with families and discrete comprehension categories.
The general reason is that, in discrete fibrations, vertical isos are identities.
In particular, as already observed, loose \gcwf{} morphisms coincide with \gcwf{} morphisms.

\begin{corollary}\label{disceqv}
The 2-category $\dsccompcatCat$ of comprehension categories with discrete fibration
and the 2-category $\cwfCat$ of categories with families are adjoint 2-equivalent.

The adjoint 2-equivalence restricts to the wide 2-full sub-2-categories on the pseudo and strict morphisms.
\end{corollary}

\begin{proof}
As shown in~\eqref{cmpct2wccmd},~\eqref{wccmd2cmpct}, and~\eqref{cor:jdtt2wccmd},
the 2-functors $\cmpcttowccmd$, $\wccmdtocmpct$, and $\jdttpstowccmd$
fix the component on $\fibCat$ of the structures involved.
The 2-functor $\wccmdtojdttps$ fixes the first component on $\fibCat$,
and its second component on $\fibCat$ is the fibration of coalgebras~\eqref{cor:wccmd2jdtt},
which is clearly discrete if the original fibration is.
Therefore all the 2-functors involved restrict to 2-functors between
$\dsccompcatCat$ and $\cwfCat$.

Note also that the invertible 2-cells $\zeta$ and $\xi$ witnessing that
$\unitadjloocmd$ is weakly invertible have vertical components~\eqref{lem:inverse_to_comparison}.
Therefore $\unitadjloocmd$ is, in fact, invertible.

Finally, note that the mate of the (vertical) natural iso $\emfrg{}\xi$
of the inverse to the unit $\unitadjloocmd$ is itself a vertical natural iso~\eqref{lem:inverse_to_comparison}.
It follows that the inverse of $\unitadjloocmd$ is also a strict \gcwf{} morphism.
Therefore the 2-equivalence restricts as required.
\end{proof}

Recall that a \define{full comprehension category} is one whose comprehension functor $\chi$ is fully faithful.

It is well-known that a functor $F \colon \ctg{C} \to \ctg{C'}$ factors
as an injective-on-objects functor $\mathrm{a}F \colon \ctg{C} \to F\!\ctg{C}$
followed by a fully faithful one $\mathrm{o}F \colon F\!\ctg{C} \to \ctg{C'}$.
The category $F\!\ctg{C}$ has the same objects of $\ctg{C}$,
and $F\!\ctg{C}(X,Y) := \ctg{C'}(FX,FY)$.
It is also well-known that this factorisation is part of an orthogonal factorisation on $\Cat$,
and that it provides a reflection of the arrow category $\Cat^2=:\catof{Fun}$ into the full sub-category of fully faithful functors.
\begin{equation}
\begin{tikzcd}[column sep=5em]\label{fullfaith-refl}
\catof{{f\textbf{\&}fFun}}	\ar[r,shift right=.5ex,hook,""{name=rg}]
&	\catof{{Fun}}
\ar[l,end anchor={[xshift=-.5ex]north east},start anchor=north west
	,shift left=1ex,bend right=6ex,""{name=lt}]
\ar[from=rg,to=lt,phantom,"\adj{-90}"{description,pos=.4}]
\end{tikzcd}
\end{equation}
Moreover, the factorisation extends to morphisms in $\Cat/\ctg{B}$,
as well as to morphisms in $\fibCat_{\ctg{B}}$,
the category of fibrations over $\ctg{B}$.
Similarly, the reflection also works replacing $\Cat^2$ with $(\Cat/\ctg{B})^2$ or $\fibCat_{\ctg{B}}^2$.
With these observations, it is possible to see that
the reflection lifts to a reflection of comprehension categories and strict morphisms into full comprehension categories and strict morphisms, see~\cite[Lemma~4.9]{comprehensioncats}
where the result is attributed to Erhard.
The reflector maps $(p,\chi)$ to its \define{heart} $(p^\heartsuit,\chi^\heartsuit)$,
where $\chi^\heartsuit = \mathrm{o}\chi$,
and $p^\heartsuit\colon \chi\ctg{E} \to \ctg{B}$
is the unique functor induced by the universal property of the unit $\mathrm{a}\chi$.

On the other hand,
it is also easy to see that the reflection in~\eqref{fullfaith-refl} lifts to a 2-reflection on functors and \emph{pseudo} morphisms.
\begin{equation}\label{fullfaith-refl-ps}
\begin{tikzcd}[column sep=5em]
\catof{\pseudosub{f\textbf{\&}fFun}}	\ar[r,shift right=.5ex,hook,""{name=rg}]
&	\catof{\pseudosub{Fun}}
\ar[l,end anchor={[xshift=-.5ex]north east},start anchor=north west
	,shift left=1ex,bend right=6ex,""{name=lt}]
\ar[from=rg,to=lt,phantom,"\adj{-90}"{description,pos=.4}]
\end{tikzcd}
\end{equation}
More precisely, $\catof{\pseudosub{Fun}}$ is the 2-category where
the 1-cells are squares commuting up to a natural iso,
and the 2-cells are the 2-cells of $\Cat^2$ compatible with the natural isos.
This construction does not seem to give a reflection
when instead of $\catof{\pseudosub{Fun}}$ one considers $\catof{Fun_{lax}}$,
where 1-cells are squares filled with an arbitrary natural transformation.

\begin{proposition}\label{fullfaith-cmpcat}
The heart of a comprehension category
lifts to a 2-reflection to the inclusion of full comprehension categories and pseudo morphisms into
comprehension categories and pseudo morphisms.
\begin{equation}\label{fullfaith-cmpcat-adj}
\begin{tikzcd}[column sep=5em]
\fllcompcatpsCat	\ar[r,shift right=.5ex,hook,""{name=rg}]
&	\compcatpsCat
\ar[l,end anchor={[xshift=-.5ex]north east},start anchor=north west
	,shift left=1ex,bend right=6ex,"{(-)^\heartsuit}"',""{name=lt}]
\ar[from=rg,to=lt,phantom,"\adj{-90}"{description,pos=.4}]
\end{tikzcd}
\end{equation}

A (split) cleavage on the fibration $p$ induces a (split) cleavage on $p^\heartsuit$,
so that the 2-reflection restricts to the full sub-2-categories on split comprehension categories.

Moreover, the 2-reflection also restricts to the 2-full sub-2-categories on split comprehension categories,
where morphisms preserve the cleavage.
\begin{equation}\label{fullfaith-cmpcat-spl}
\begin{tikzcd}[column sep=5em]
\fllsplcompcatpsCat	\ar[r,shift right=.5ex,hook,""{name=rg}]
&	\splcompcatpsCat
\ar[l,end anchor={[xshift=-.5ex]north east},start anchor=north west
	,shift left=1ex,bend right=6ex,"{(-)^\heartsuit}"',""{name=lt}]
\ar[from=rg,to=lt,phantom,"\adj{-90}"{description,pos=.4}]
\end{tikzcd}
\end{equation}

All these 2-reflections restrict to the sub-2-categories on strict morphisms.
\end{proposition}

\begin{proof}
Note that $\compcatpsCat((p,\chi),(p',\chi'))$ is the limit in $\Cat$ of the diagram of forgetful functors below
\[\begin{tikzcd}[row sep=2em]
\fibCat(p,p')	\ar[d] \ar[drr,bend left=1em]
&&	\catof{\pseudosub{Fun}}(\chi,\chi')	\ar[dl,bend right=1.5em,crossing over] \ar[d]
\\
\Cat(\ctg{B},\ctg{B}')	\ar[r,"{(-)^2}"]
&	\Cat(\ctg{B}^2,\ctg{B}'^2)
&	\Cat(\ctg{E},\ctg{E}')
\end{tikzcd}\]

When $\chi'$ is fully faithful, the functor
$(-)\circ \mathrm{a}\chi \colon \catof{\pseudosub{f\textbf{\&}fFun}}(\chi^\heartsuit,\chi')
\to \catof{\pseudosub{Fun}}(\chi,\chi')$
is invertible by~\eqref{fullfaith-refl-ps}.
It is also follows that the functor
$(-)\circ \mathrm{a}\chi \colon \fibCat(p^\heartsuit,p') \to \fibCat(p,p')$
is invertible.
Therefore
\[\begin{tikzcd}[column sep=4em]
\fllcompcatpsCat((p^\heartsuit,\chi^\heartsuit),(p',\chi'))
\ar[r,"\sim","(-)\circ \mathrm{a}\chi"']
&	\compcatpsCat((p,\chi),(p',\chi'))
\end{tikzcd}\]
as required.

A cartesian lift for $p^\heartsuit$ is the image under $\chi$ of a cartesian lift for $p$.
This choice is split since $\chi$ is a functor.
The claim that the induced 1-cells preserve the cleavage has a straightforward verification.
\end{proof}

It is well known, and easy to verify,
that discrete fibrations are 2-coreflective in split fibrations.
The coreflector maps a fibration
to its wide subfibration with only the arrows of the cleavage.
The total category is indeed a category since the cleavage is split,
and the fibration is clearly discrete.
The 2-coreflection lifts to a 2-coreflection between
discrete comprehension categories and split comprehension categories
\begin{equation}\label{disc-split}
\begin{tikzcd}[column sep=5em]
\dsccompcatCat	\ar[r,shift right=.5ex,hook,""{name=la}]
&	\splcompcatCat
\ar[l,start anchor={[xshift=.2ex,yshift=-.5ex]north west}
	 ,end anchor={[xshift=-.5ex,yshift=-.5ex]north east}
	 ,bend right=2em
	 ,"{|-|}"'
	 ,""{name=ra}]
\ar[from=la,to=ra,phantom,"\adj{90}"{description}]
\end{tikzcd}
\end{equation}
which restricts to subcategories on pseudo and strict morphisms.

By composing the adjunctions in~\eqref{disc-split} and in~\eqref{fullfaith-cmpcat-spl},
one obtains the right-hand 2-equivalence below.
The left-hand one is from~\eqref{disceqv}.
\begin{equation}\label{eqv-cwf-fullspl}
\cwfpsCat \equiv \dsccompcatpsCat \equiv \fllsplcompcatpsCat
\end{equation}
When restricted to strict morphisms, it is the equivalence in~\cite[Proposition~1.2.4]{BlancoJ}.

\begin{remark}\label{ps-cwfmorph}
Note that the 1-cells in the 2-categories in~\eqref{eqv-cwf-fullspl}
involve functors that preserve the cleavage,
since morphisms in $\fibCat$ between discrete fibrations must preserve the (split) cleavage,
as those are the only (cartesian) arrows.
To obtain more morphisms, one should use~\eqref{fullfaith-cmpcat-adj} instead of~\eqref{fullfaith-cmpcat-spl} and,
given categories with families $(u,\du)$ and $(u',\du')$,
look at
\begin{equation}\label{pscwfmorph}
\fllcompcatpsCat(\gcwftocomp(u,\du)^\heartsuit,\gcwftocomp(u',\du')^\heartsuit).
\end{equation}
Note that those in the image of $(-)^\heartsuit$ do preserve the cleavage,
but the others do not (necessarily).
If we also require the base categories to have terminal objects preserved by the morphisms
then the morphisms in~\eqref{pscwfmorph} are the \emph{pseudo cwf-morphisms} between $(u,\du)$ and $(u',\du')$
of Clairambault and Dybjer~\citey[Definition~3.1]{ClairambaultDybjer}.

Indeed, consider first the functor $\textbf{\textit{T}}\colon \ctg{C}\opp \to \Cat$ from~\cite[Proposition~2.7]{ClairambaultDybjer}
associated to a category with families $(\ctg{C},T \colon \ctg{C}\opp \to \mathsf{Fam})$.
The (split) fibration $\pi_{\textbf{\textit{T}}} \colon \int \textbf{\textit{T}} \to \ctg{C}$ corresponding to
$\textbf{\textit{T}} \colon \ctg{C}\opp \to \Cat$ under the Grothendieck construction
is the underlying fibration of the heart $(\gcwftocomp(\pi_{\mathrm{Ty}_T},\pi_{\mathrm{Tm}_T}))^\heartsuit$
of the comprehension category associated to $(\ctg{C},T)$,
where $(\pi_{\mathrm{Ty}_T},\pi_{\mathrm{Tm}_T})$ is the \gencwf{y} described in~\eqref{exa:free_syntactic}.

As observed in~\cite{ClairambaultDybjer}, a pseudo cwf-morphism $(F,\sigma)\colon$ $(\ctg{C},T) \to (\ctg{C}',T')$
induces a morphism of fibrations $(F,H)\colon \pi_{\mathrm{Ty}_T} \to \pi_{\mathrm{Ty}_{T'}}$ which preserves
context comprehension up a natural iso $\rho$.
This means precisely that $(F,H,\rho)$ is an object in~\eqref{pscwfmorph}.

Conversely, given an object $(F,H,\zeta)$ in~\eqref{pscwfmorph},
the isomorphism $\theta$ is given by the fact that $H^\heartsuit$ preserves cartesian arrows.
It ``preserves substitution in terms'' since postcomposition with $\theta$ preserves sections of display maps.
The ``coherence conditions'' involving $\theta$ correspond to the fact that the cleavage of the heart of a cwf is split.
The iso $\rho$ witnessing the preservation of context comprehension is $\zeta$ itself.
\end{remark}

\bibliography{thebib}

\begin{thebibliography}{29}
\makeatletter
\newcommand{\dinatlabel}[1]%
{\ifNAT@numbers\else\NAT@biblabelnum{#1}\hspace{2\labelsep}\fi}
\makeatother
\expandafter\ifx\csname natexlab\endcsname\relax\def\natexlab#1{#1}\fi
\expandafter\ifx\csname url\endcsname\relax\def\url#1{\texttt{#1}}\fi

\bibitem[Awodey(2018)]{awodey_2018}
\dinatlabel{Awodey 2018} \textsc{Awodey}, Steve:
\newblock Natural models of homotopy type theory.
\newblock In: \emph{Mathematical Structures in Computer Science}
\newblock 28 (2018), Nr.~2, S.~241–286

\bibitem[Blanco(1991)]{BlancoJ}
\dinatlabel{Blanco 1991} \textsc{Blanco}, Javier:
\newblock \emph{Relating categorical approaches to type dependency}, Katholieke
  Universiteit Nijmegen, Diplomarbeit, 1991

\bibitem[Brown(1973)]{Brown1973}
\dinatlabel{Brown 1973} \textsc{Brown}, Kenneth~S.:
\newblock Abstract Homotopy Theory and Generalized Sheaf Cohomology.
\newblock In: \emph{Trans.\ Amer.\ Math.\ Soc.}
\newblock 186 (1973), S.~419--458

\bibitem[Cartmell(1986)]{CARTMELL1986209}
\dinatlabel{Cartmell 1986} \textsc{Cartmell}, John:
\newblock Generalised algebraic theories and contextual categories.
\newblock In: \emph{Annals of Pure and Applied Logic}
\newblock 32 (1986), S.~209--243

\bibitem[Clairambault und Dybjer(2014)]{ClairambaultDybjer}
\dinatlabel{Clairambault und Dybjer 2014} \textsc{Clairambault}, Pierre~;
  \textsc{Dybjer}, Peter:
\newblock {The Biequivalence of Locally Cartesian Closed Categories and
  Martin-L{\"o}f Type Theories}.
\newblock In: \emph{{Mathematical Structures in Computer Science}}
\newblock 24 (2014), April, Nr.~05, S.~e240501

\bibitem[Coraglia und Di~Liberti(2022)]{cjd}
\dinatlabel{Coraglia und Di~Liberti 2022} \textsc{Coraglia}, Greta~;
  \textsc{Di~Liberti}, Ivan:
\newblock \emph{Context, Judgement, Deduction}.
\newblock Available as
  \href{https://arxiv.org/abs/2111.09438}{\texttt{arXiv:2111.09438}}.
\newblock 2022

\bibitem[Coraglia und Emmenegger(2023)]{ce_subtyping}
\dinatlabel{Coraglia und Emmenegger 2023} \textsc{Coraglia}, Greta~;
  \textsc{Emmenegger}, Jacopo:
\newblock \emph{Categorical models of subtyping}.
\newblock Accepted for publication in \emph{Post Proceedings of TYPES2023},
  LIPIcs. Preprint available as
  \href{https://arxiv.org/abs/2312.14600}{\texttt{arXiv:2312.14600}}.
\newblock 2023

\bibitem[Dubuc(1970)]{Dubuc1970}
\dinatlabel{Dubuc 1970} \textsc{Dubuc}, Eduardo:
\newblock \emph{Lecture Notes in Mathematics}. Bd. 145: \emph{{Kan Extensions
  in Enriched Category Theory}}.
\newblock Berlin-New York~: Springer-Verlag, 1970. --
\newblock xvi+173~S

\bibitem[Dybjer(1996)]{dybjer_inttt}
\dinatlabel{Dybjer 1996} \textsc{Dybjer}, Peter:
\newblock Internal type theory.
\newblock In: \textsc{Berardi}, Stefano (Hrsg.)~; \textsc{Coppo}, Mario
  (Hrsg.): \emph{Types for Proofs and Programs}.
\newblock Berlin, Heidelberg~: Springer Berlin Heidelberg, 1996, S.~120--134

\bibitem[Ehrhard(1988)]{ehrhard}
\dinatlabel{Ehrhard 1988} \textsc{Ehrhard}, Thomas:
\newblock A categorical semantics of constructions.
\newblock In: \emph{Proceedings. Third Annual Symposium on Logic in Computer
  Science}, 1988, S.~264--273

\bibitem[Eilenberg und Moore(1965)]{em65}
\dinatlabel{Eilenberg und Moore 1965} \textsc{Eilenberg}, Samuel~;
  \textsc{Moore}, John~C.:
\newblock {Adjoint functors and triples}.
\newblock In: \emph{Illinois Journal of Mathematics}
\newblock 9 (1965), Nr.~3, S.~381 -- 398

\bibitem[Fiore(2008)]{fiore_att}
\dinatlabel{Fiore 2008} \textsc{Fiore}, Marcelo:
\newblock \emph{Notes on Algebraic Type Theory}.
\newblock Available at \url{https://www.cl.cam.ac.uk/~mpf23/Notes/att.pdf}.
\newblock 2008

\bibitem[Hofmann(1997)]{hofmann_1997}
\dinatlabel{Hofmann 1997} \textsc{Hofmann}, Martin:
\newblock Syntax and Semantics of Dependent Types.
\newblock In: \textsc{Pitts}, Andrew~M. (Hrsg.)~; \textsc{Dybjer}, P.Editors
  (Hrsg.): \emph{Semantics and Logics of Computation}.
\newblock Cambridge University Press, 1997
\newblock (Publications of the Newton Institute), S.~79–130

\bibitem[Huber(1961)]{huber61}
\dinatlabel{Huber 1961} \textsc{Huber}, Peter~J.:
\newblock Homotopy theory in general categories.
\newblock In: \emph{Mathematische Annalen}
\newblock (1961). --
\newblock URL \url{https://doi.org/10.1007/BF01396534}

\bibitem[Jacobs(1993)]{comprehensioncats}
\dinatlabel{Jacobs 1993} \textsc{Jacobs}, Bart:
\newblock Comprehension categories and the semantics of type dependency.
\newblock In: \emph{Theoretical Computer Science}
\newblock 107 (1993), Nr.~2, S.~169--207. --
\newblock ISSN 0304-3975

\bibitem[Jacobs(1999)]{jacobs1999categorical}
\dinatlabel{Jacobs 1999} \textsc{Jacobs}, Bart:
\newblock \emph{Categorical logic and type theory}.
\newblock Elsevier, 1999

\bibitem[Kelly und Street(1974)]{KellyStreetReview2cats}
\dinatlabel{Kelly und Street 1974} \textsc{Kelly}, Gregory~M.~;
  \textsc{Street}, Ross:
\newblock Review of the elements of 2-categories.
\newblock In: \textsc{Kelly}, G.~M. (Hrsg.): \emph{Category Seminar}.
\newblock Berlin, Heidelberg~: Springer Berlin Heidelberg, 1974, S.~75--103. --
\newblock ISBN 978-3-540-37270-7

\bibitem[Kleisli(1965)]{kleisli65}
\dinatlabel{Kleisli 1965} \textsc{Kleisli}, Heinrich:
\newblock Every Standard Construction is Induced by a Pair of Adjoint Functors.
\newblock In: \emph{Proceedings of the American Mathematical Society}
\newblock 16 (1965), Nr.~3, S.~544--546. --
\newblock URL \url{http://www.jstor.org/stable/2034693}. -- Zugriffsdatum:
  2023-02-14. --
\newblock ISSN 00029939, 10886826

\bibitem[Larrea(2018)]{larrea2018}
\dinatlabel{Larrea 2018} \textsc{Larrea}, Marco~F.:
\newblock \emph{{Models of Dependent Type Theory from Algebraic Weak
  Factorisation Systems}}, University of Leeds, Dissertation, 2018

\bibitem[Lawvere(1970)]{LawvereF:equhcs}
\dinatlabel{Lawvere 1970} \textsc{Lawvere}, F.~W.:
\newblock Equality in hyperdoctrines and comprehension schema as an adjoint
  functor.
\newblock In: \textsc{Heller}, A. (Hrsg.): \emph{Proc. {N}ew {Y}ork {S}ymposium
  on {A}pplication of {C}ategorical {A}lgebra}, Amer.{M}ath.{S}oc., 1970,
  S.~1--14

\bibitem[MacLane(1978)]{CWM}
\dinatlabel{MacLane 1978} \textsc{MacLane}, Saunders:
\newblock \emph{Categories for the Working Mathematician}.
\newblock Springer, 1978

\bibitem[Martin{-}L{\"{o}}f(1984)]{mltt}
\dinatlabel{Martin{-}L{\"{o}}f 1984} \textsc{Martin{-}L{\"{o}}f}, Per:
\newblock \emph{Studies in proof theory}. Bd.~1: \emph{Intuitionistic type
  theory}.
\newblock Bibliopolis, 1984. --
\newblock ISBN 978-88-7088-228-5

\bibitem[Melli\`{e}s und Rolland(2020)]{melliesrolland2020}
\dinatlabel{Melli\`{e}s und Rolland 2020} \textsc{Melli\`{e}s},
  Paul-Andr\'{e}~; \textsc{Rolland}, Nicolas:
\newblock {Comprehension and Quotient Structures in the Language of
  2-Categories}.
\newblock In: \textsc{Ariola}, Zena~M. (Hrsg.): \emph{5th International
  Conference on Formal Structures for Computation and Deduction (FSCD 2020)}
  Bd.~167.
\newblock Dagstuhl, Germany~: Schloss Dagstuhl -- Leibniz-Zentrum f{\"u}r
  Informatik, 2020, S.~6:1--6:18

\bibitem[Moggi(1991)]{moggi_1991}
\dinatlabel{Moggi 1991} \textsc{Moggi}, Eugenio:
\newblock A category-theoretic account of program modules.
\newblock In: \emph{Mathematical Structures in Computer Science}
\newblock 1 (1991), Nr.~1, S.~103–139

\bibitem[Palmgren(2019)]{PALMGREN2019102715}
\dinatlabel{Palmgren 2019} \textsc{Palmgren}, Erik:
\newblock Categories with families and first-order logic with dependent sorts.
\newblock In: \emph{Annals of Pure and Applied Logic}
\newblock 170 (2019), Nr.~12, S.~102715. --
\newblock ISSN 0168-0072

\bibitem[Rijke(2022)]{Rijke_intro}
\dinatlabel{Rijke 2022} \textsc{Rijke}, Egbert:
\newblock \emph{Introduction to Homotopy Type Theory}.
\newblock 2022. --
\newblock {To appear in \emph{Cambridge Studies in Advanced Mathematics},
  Cambridge University Press. Available as
  \href{https://arxiv.org/abs/2212.11082}{\texttt{arXiv:2212.11082}}}

\bibitem[Street(1972)]{street72formnd}
\dinatlabel{Street 1972} \textsc{Street}, Ross:
\newblock The formal theory of monads.
\newblock In: \emph{Journal of Pure and Applied Algebra}
\newblock 2 (1972), Nr.~2, S.~149--168

\bibitem[Streicher(2022)]{streicher2022fibred}
\dinatlabel{Streicher 2022} \textsc{Streicher}, Thomas:
\newblock \emph{Fibred categories {\`a} la {Jean} {B{\'e}nabou}}.
\newblock Available as
  \href{https://arxiv.org/abs/1801.02927}{\texttt{arXiv:1801.02927}}.
\newblock 2022

\bibitem[Uemura(2023)]{Uemura2023}
\dinatlabel{Uemura 2023} \textsc{Uemura}, Taichi:
\newblock A general framework for the semantics of type theory.
\newblock In: \emph{Mathematical Structures in Computer Science}
\newblock 33 (2023), Nr.~3, S.~134–179. --
\newblock Available as
  \href{https://arxiv.org/abs/1904.04097}{\texttt{arXiv:1904.04097}}

\end{thebibliography}
\bibliographystyle{dinat}

\end{document}